\documentclass[a4paper, 10pt]{amsart}
\usepackage{graphicx}
\newtheorem{theorem}{Theorem}[section]

\newtheorem{lemma}[theorem]{Lemma}
\newtheorem{proposition}[theorem]{Proposition}
\newtheorem{definition}[theorem]{Definition}
\newtheorem{remark}[theorem]{Remark}

\newtheorem{question}[theorem]{Question}

\begin{document}

\title[Second main theorem with moving hypersurfaces]{Second main theorem for holomorphic curves into algebraic varieties intersecting moving hypersurfaces targets}

\author[L. B. Xie]{Libing Xie}
\address[Libing Xie]{Department of Mathematics, Nanchang University, Jiangxi 330031, P. R. China}
\email{xielibing123@126.com}

\author[T. B. Cao]{Tingbin Cao}
\address[Tingbin Cao]{Department of Mathematics, Nanchang University, Jiangxi 330031, P. R. China}
\email{tbcao@ncu.edu.cn (the corresponding author)}

\thanks{This paper was supported by the National Natural Science Foundation of China (\#11871260, \#11461042), the outstanding young talent assistance program of Jiangxi Province (\#20171BCB23002) in China.}
\date{}



\begin{abstract}Since the great work on holomorphic curves into algebraic varieties intersecting hypersurfaces in general position established by Ru in 2009, recently there has been some developments on the second main theorem into algebraic varieties intersecting moving hypersurfaces targets. The main purpose of this paper is to give some interesting improvements of Ru's second main theorem for moving hypersurfaces targets located in subgeneral position with index.
\end{abstract}

\maketitle

\section{Introduction and main results}
It is well-known that in 1933, H. Cartan established Nevanlinna theory for meromorphic functions to the case of linearly nondegenerate holomorphic curves into complex projective spaces intersecting hyperplanes in general position, and conjectured that it is still true for moving hyperplanes targets. From then on, higher dimensional Nevanlinna theory has been studied very hot (refer to \cite{Ru3, V, noguchi-winkelmann}). In 2009, Ru \cite{Ru4} proposed a great work on second main theorem of algebraically nondegenerate holomorphic curves into smooth complex varieties intersecting hypersurfaces in general position, which is a generalization of the Cartan's second main theorem and his own former result \cite{Ru1} completely solving the Shiffman's conjecture\cite{S} under the motivation of Corvaja-Zannier \cite{So} in Diophantine approximation.\par

Thus, it is natural and interesting to investigate the Ru's second main theorem into complex projective spaces and even into complex algebraic varieties for the moving hypersurfaces targets. Based on their affirmation of the Shiffman's conjecture for moving hypersurfaces targets \cite{D}, recently, Dethloff and Tan \cite{C} continue to prove successfully the following theorem. In the special case where the coefficients of the polynomials $Q_j$'s are constant and the variety $V$ is smooth, this is the Ru's second main theorem \cite{Ru4}.\par

\begin{theorem} \cite{C} \label{TA} Let $V\subset \mathbb{P}^{n}(\mathbb{C})$ be an irreducible (possibly singular) variety of dimension $\ell,$ and let $f$ be a nonconstant holomorphic map of $\mathbb{C}$ into $V.$  and let $\mathcal{D}=\{D_{1},\ldots,D_{q}\}$ be a family of slowly moving hypersurfaces (with respect to $f$) in general position, and let $\mathcal{Q}=\{Q_{1}, \ldots, Q_{q}\}$ be the set of the set of the defining homogeneous polynomials of $\mathcal{D}$ with $\deg Q_{j}=d_{j},(j=1,...,q)$ and $Q_{j}(f)\not\equiv 0$ for $j=1, \ldots, q.$ Assume that $f$ is algebraically nondegenerate over $\mathcal{K}_{\mathcal{Q}}.$ Then, for any $\epsilon >0$,
\begin{eqnarray}\label{E1.1}
\sum^{q}_{j=1}(1/d_{j})m_{f}(r, D_{j})\leq(\ell+1+\epsilon)T_{f}(r)
\end{eqnarray}  holds for all $r$ outside a set with finite Lebesgue measure.\end{theorem}

For the special case $V=\mathbb{P}^{n}(\mathbb{C}),$  S. D. Quang \cite{Q1} recently gave a second main theorem with truncated counting functions for meromorphic mappings into $\mathbb{P}^{n}(\mathbb{C})$ intersecting a family of moving hypersurfaces in subgeneral position , which can be possibly good at the uniqueness problems of meromorphic mappings.\par

\begin{theorem}\cite{Q1}\label{TC}
Let $f$ be a nonconstant meromorphic map of $\mathbb{C}^{m}$ into $\mathbb{P}^{n}(\mathbb{C}).$ Let ${\{Q_{i}\}_{i=1}^{q}}$ be a collection of slowly moving hypersurfaces in $N$-subgeneral position with $\deg Q_{j}=d_{j}(1\leq j\leq q).$ Assume that $f$ is algebraically nondegenerate over $\mathcal{K}_{\mathcal{Q}}.$  Then, for any $\epsilon >0,$
\begin{eqnarray}\label{E1.2}
(q-(N-n+1)(n+1)-\epsilon)T_{f}(r)\leq\sum^{q}_{i=1}\frac{1}{d_{j}}N^{[L_0]}(r,f^{*}Q_i)+o(T_{f}(r))
\end{eqnarray}holds for all $r$ outside a set with finite Lebesgue measure, where
$$L_{0}:={L+n \choose n}p_{0}^{{L+n \choose n}{L+n \choose n}-1){q \choose n}-2}-1,$$
with $$L:=(n +1)d +2(N-n+1)(n +1)^{3}I(\epsilon^{-1})d,$$ $d := lcm(d_{1},\ldots, d_{q})$ is the least common multiple of all $\{d_{i}\},$
and $$p_{0}:=\left[\frac{{L+n \choose n}({L+n \choose n}-1){q \choose n}-1}{\log(1+\frac{\epsilon}{3(n+1)(N-n+1)})}\right]^{2}.$$
\end{theorem}\par

In this paper, we mainly combine their methods \cite{C,Q1, Y} together and adopt the new concept of the index of subgeneral position due to Ji-Yan-Yu \cite{J2} to obtain some interesting developments of Ru's second main theorem for moving hypersurfaces targets, which are improvements and extensions of Theorem \ref{TA} and Theorem \ref{TC}.\par

According to \cite{J2}, we can give a similar definition for moving hypersurfaces located in $m$-subgeneral position with index $k.$\par

\begin{definition}\label{D1.1} Let $V$ be an algebraic subvariety of $\mathbb{P}^{n}(\mathbb{C}).$  Let $\{D_{1},\ldots,D_{q}\}$ be a family of moving hypersurfaces of $\mathbb{P}^{n}(\mathbb{C}).$ Let $N$ and $\kappa$ be two positive integers such that $N\geq\dim V\geq \kappa.$\par

(a). The hypersurfaces $\{D_{1},\ldots,D_{q}\}$  are said to be in general position (or say in weakly general position) in $V$ if there exists $z\in\mathbb{C}^{m}$ (if this condition is satisfied for one $z\in\mathbb{C}^{m}$, it is also satisfied for all $z$ except for a discrete set and the analytic set $I(f)=\{f_{0}=\cdots=f_{n}=0\}$ of codimension $\geq 2$) for any subset $I\subset\{1,\ldots,q\}$ with $\sharp I\leq \dim V+1,$
$$codim\left(\bigcap_{i\in I}D_{i}(z)\cap V\right)\geq\sharp I.$$\par

(b). The hypersurfaces $\{D_{1},\ldots,D_{q}\}$  are said to be in $N$-subgeneral position in $V$ if there exists $z\in\mathbb{C}^{m}$ (if this condition is satisfied for one $z\in\mathbb{C}^{m}$, it is also satisfied for all $z$ except for a discrete set and the analytic set $I(f)=\{f_{0}=\cdots=f_{n}=0\}$ of codimension $\geq 2$) for any subset $I\subset\{1,\ldots,q\}$ with $\sharp I\leq N+1,$
$$\dim\left(\bigcap_{i\in I}D_{i}(z)\cap V\right)\leq N-\sharp I.$$\par

(c). The hypersurfaces $\{D_{1},\ldots,D_{q}\}$  are said to be in $N$-general position with index $\kappa$ in $V$ if $D_{1},\ldots,D_{q}$ are in $N$-subgeneral position and if there exists $z\in\mathbb{C}^{m}$ (if this condition is satisfied for one $z\in\mathbb{C}^{m}$, it is also satisfied for all $z$ except for a discrete set and the analytic set $I(f)=\{f_{0}=\cdots=f_{n}=0\}$ of codimension $\geq 2$) for any subset $I\subset\{1,\ldots,q\}$ with $\sharp I\leq \kappa,$
$$codim\left(\bigcap_{i\in I}D_{i}(z)\cap V\right)\geq\sharp I.$$
(Here we set $\dim \emptyset=-\infty$.)
\end{definition}

Now, we state our main result which are an improvement and extension of the above two theorems concerning moving hypersurfaces targets located in subgeneral position with index. Theorem \ref{TA} is just the following result for the special case whenever $N=\dim V$ and $\kappa=1.$ \par

\begin{theorem}\label{T1.1}
Let $f:\mathbb{C}^{m}\rightarrow V\subset\mathbb{P}^{n}(\mathbb{C})$ be a nonconstant meromorphic map, where $V$ is an irreducible algebraic subvariety of dimension $\ell.$ Let $\mathcal{Q}=\{Q_{1},\ldots,Q_{q}\}$ be a collection of slowly moving hypersurfaces in $N$-subgeneral position with index $\kappa$ in $V,$ and $\deg Q_{j}=d_{j}(j=1, \ldots, q,).$ Assume that $f:\mathbb{C}^{m}\rightarrow V$ is algebraically nondegenerate over $\mathcal{K}_{\mathcal{Q}}.$ Then, for any $\epsilon >0,$
\begin{eqnarray}\label{E1.13}
&&\left(q-(1+\frac{N-\ell}{\max\{1,\min\{N-\ell,\kappa\}\}})(\ell+1)-\epsilon\right)T_{f}(r)\\\nonumber&\leq&\sum^{q}_{i=1}\frac{1}{d_{j}}N(r,f^{*}Q_i)+o(T_{f}(r)),
\end{eqnarray}
holds for all $r$ outside a set with finite Lebesgue measure
\end{theorem}

When $V=P^{n}(\mathbb{C}),$ we can have the following second main theorem with truncation, and thus Theorem \ref{TC} is just the special case whenever $\kappa=1.$\par

\begin{theorem}\label{T1.2}
Let $f$ be a nonconstant meromorphic map of $\mathbb{C}^{m}$ into $\mathbb{P}^{n}(\mathbb{C}).$ Let ${\{Q_{i}\}_{i=1}^{q}}$ be a collection of slowly moving hypersurfaces in $N$-subgeneral position with index $\kappa,$ and $\deg Q_{j}=d_{j}(1\leq j\leq q).$ Assume that $f$ is algebraically nondegenerate over $\mathcal{K}_{\mathcal{Q}}.$ Then, for any $\epsilon >0,$
\begin{eqnarray}\label{E1.13}
&&\left(q-(1+\frac{N-n}{\max\{1,\min\{N-n,\kappa\}\}})(n+1)-\epsilon\right)T_{f}(r)\\\nonumber&\leq&\sum^{q}_{i=1}\frac{1}{d_{j}}N^{[L_0]}(r,f^{*}Q_i)+o(T_{f}(r)).
\end{eqnarray}
holds for all $r$ outside a set with finite Lebesgue measure, where
\begin{eqnarray*}
L_{0}:={L+n \choose n}p_{0}^{{L+n \choose n}{L+n \choose n}-1){q \choose n}-2}-1.
\end{eqnarray*}with
$$L:=(n +1)d +2(1+\frac{N-n}{\max\{1,\min\{N-n,\kappa\}\}})(n +1)^{3}I(\epsilon^{-1})d,$$ $d:= lcm(d_{1},\ldots, d_{q})$ is the least common multiple of all $\{d_{i}\},$ and $$p_{0}:=\left[\frac{{L+n \choose n}({L+n \choose n}-1){q \choose n}-1}{\log(1+\frac{\epsilon}{3(n+1)(1+\frac{N-n}{\max\{1,\min\{N-n,\kappa\}\}})})}\right]^{2}.$$
\end{theorem}\par

Remark that very recently, Yan and Yu \cite{Y} considered the nonconstant holomorphic curve from $\mathbb{C}$ into $\mathbb{P}^{n}(\mathbb{C})$ instead of algebraically nondegenerate and improved Theorem \ref{TC} without truncation. Thus it is interesting to ask the following question£º\par

\begin{question} \emph{In Theorem \ref{T1.1} or Theorem \ref{T1.2}, is it possible to obtain a second main theorem if the condition ``$f$ is algebraically nondegenerate over $\mathcal{K}_{\mathcal{Q}}$" is omitted?}\end{question}

The remainder is the organization as follows. In the next section, we introduce some basic notions and auxiliary results from Nevanlinna theory. Section 3 and Section 4 are the proofs of Theorem \ref{T1.1} and Theorem \ref{T1.2}, respecitively, in which the methods to dealing with moving targets by  Dethloff-Tan \cite{C},  Yan-Yu \cite{Y}, and the techniques to deal with hypersurfaces in subgeneral position instead of Nochka's weights owing to Quang \cite{Q1, Q2} are used in this paper.

\section{Basic notions and auxiliary results from Nevanlinna theory}
\subsection{The first main theorem in Nevanlinna theory}

We set $\|z\|=(\|z_{1}\|^{2}+\cdots+\|z_{m}\|^{2})^{1 / 2}$ for $z=(z_{1}, \ldots, z_{m}) \in \mathbb{C}^{m}$ and define
$$B(r):=\{z \in \mathbb{C}^{m}:\|z\|<r\}, \quad S(r):=\{z \in \mathbb{C}^{m}:\|z\|=r\}(0<r<\infty).$$\\
Define

$${v_{m-1}(z):=\left(d d^{c}\|z\|^{2}\right)^{m-1} \quad \text { and }}$$  $${\sigma_{m}(z):=d^{c} \log \|z\|^{2} \wedge\left(d d^{c} \log \|z\|^{2}\right)^{m-1} \quad \text { on } \quad \mathbf{C}^{m} \backslash\{0\}}.$$

Let $F$ be a nonzero meromorphic function on a domain $\Omega$ in $\mathbb{C}^{m} .$ For a set $\alpha=\left(\alpha_{1}, \ldots, \alpha_{m}\right)$ of nonnegative integers, we set $|\alpha|=\alpha_{1}+\ldots+\alpha_{m}$ and
$$\mathcal{D}^{\alpha} F=\frac{\partial^{|\alpha|} F}{\partial^{\alpha_{1}} z_{1} \ldots \partial^{\alpha_{m}} z_{m}}.$$
We denote by $\nu_{F}^{0}, \nu_{F}^{\infty}$ and $\nu_{F}$ the zero divisor, the pole divisor, and the divisor of the meromorphic function
$F$ respectively.\par
For a divisor $\nu$ on $\mathbb{C}^{m}$ and for a positive integer $M$ or $M=\infty,$ we set
$$\nu^{[M]}(z)=\min \{M, \nu(z)\},$$
$$n(t)=\left\{\begin{array}{ll}{\int_{|\nu| \cap B(t)} \nu(z) v_{m-1}} & {\text { if } m \geq 2} \\ {\sum_{|z| \leq t} \nu(z)} & {\text { if } m=1}.\end{array}\right.$$

The counting function of $\nu$ is defined by
$$N(r, \nu)=\int_{1}^{r} \frac{n(t)}{t^{2 m-1}} d t \quad(1<r<\infty).$$
Similarly, we define $N\left(r, \nu^{[M]}\right)$ and denote it by $N^{[M]}(r, \nu).$

Let $\varphi$ is a nonzero meromorphic function on $\mathbb{C}^{m} $. Define
$$N_{\varphi}(r)=N\left(r, \nu_{\varphi}^{0}\right), N_{\varphi}^{[M]}(r)=N^{[M]}\left(r, \nu_{\varphi}^{0}\right).$$
For brevity we will omit the character $[M]$ if $M=\infty .$\par

Let $f: \mathbb{C}^{m} \longrightarrow \mathbb{P}^{n}(\mathbb{C})$ be a meromorphic mapping. For arbitrarily fixed homogeneous coordinates
$\left(w_{0}: \cdots: w_{n}\right)$ on $\mathbb{P}^{n}(\mathbb{C}),$ we take a reduced representation $\tilde{f}=\left(f_{0}, \ldots, f_{n}\right),$ which means that each $f_{i}$ is a holomorphic function on $\mathbb{C}^{m}$ and $f(z)=\left(f_{0}(z): \cdots ; f_{n}(z)\right)$ outside the analytic set $I(f)=\{f_{0}=\cdots=f_{n}=0\}$ of codimension $\geq 2 .$ Set $\|\tilde{f}\|=(\|f_{0}\|^{2}+\cdots+\|f_{n}\|^{2})^{1 / 2} .$ The characteristic function of $f$ is defined by
$$T_{f}(r)=\int_{S(r)} \log \|\tilde{f}\| \sigma_{m}-\int_{S(1)} \log \|\tilde{f}\| \sigma_{m}.$$\par

Let $f$ and $Q$ be as above. The proximity function of $f$ with respect to $Q,$ denoted by $m_{f}(r,Q),$ is defined by
$$m_{f}(r,Q)=\int_{S(r)} \lambda_{D}(\tilde{f}) \sigma_{m}-\int_{S(1)} \lambda_{D}(\tilde{f}) \sigma_{m},$$
where $Q(\tilde{f}) = Q(f_{0}, \ldots, f_{n}),$ and $\lambda_{Q}(\tilde{f})=\log\frac{\|\tilde{f}\|^{d}\|Q\|}{|Q(\tilde{f})|}$ is the Weil function and $\|Q\|=\max_{I\in{\mathcal{I}_{d}}}\{|a_{I}|\},$ where $\mathcal{I}_{d}:=\{I=(i_{0},\ldots,i_{n})\in\mathbb{Z}^{n+1}_{\geq 0}; i_{0}+\ldots+i_{n}=d\}.$ This definition is independent of the choice of the reduced representation of $f.$\par

We denote by $f^{*}Q$ the pullback of the divisor $Q$ by $f.$ We may see that $f^{*}Q$ identifies with the zero divisor $\nu^{0}_{Q(\tilde{f})}$ of the function $Q(\tilde{f}).$ By Jensen¡¯s fomular, we have
$$N_{Q(\tilde{f})}(r)=\int_{S(r)} \log |Q(\tilde{f})| \sigma_{m}-\int_{S(1)} \log |Q(\tilde{f})| \sigma_{m}.$$ For convenience, we will denote $N(r,f^{*}Q)=N_{Q(f)}(r).$\par

\begin{theorem}(First Main Theorem).\cite{Q2}
Let $f:\mathbb{C}^{m} \longrightarrow \mathbb{P}^{n}(\mathbb{C})$ be a holomorphic map, and let $Q$ be a hypersurface in $\mathbb{P}^{n}(\mathbb{C})$ of degree $d.$ If $f(\mathbb{C}^{m}) \not\subset Q,$ then for every real number $r$ with $0 < r < +\infty,$
$$dT_{f}(r)=m_{f}(r,Q)+N_{Q(f)}(r)+O(1),$$ where $O(1)$ is a constant independent of $r.$
\end{theorem}

Let $\varphi$ be a nonzero meromorphic function on $\mathbb{C}^{m},$ which are occasionally regarded as a meromorphic map into $\mathbb{P}^{1}(\mathbb{C}) .$ The proximity function of $\varphi$ is defined by $$m(r, \varphi):=\int_{S(r)} \log \max (|\varphi|, 1) \sigma_{m}.$$
The Nevanlinna's characteristic function of $\varphi$ is defined as follows $$T(r, \varphi):=N_{\frac{1}{\varphi}}(r)+m(r, \varphi).$$
Then $$T_{\varphi}(r)=T(r, \varphi)+O(1).$$
The function $\varphi$  is said to be small (with respect to f) if $\| T_{\varphi}(r)=o\left(T_{f}(r)\right),$ where here the notion $\|$ means that the property holds possibly outside a set with finite Lebesgue measure.\par

We denote by $\left.\mathcal{M} \text { (resp. } \mathcal{K}_{f}\right)$ the field of all meromorphic functions (resp. small meromorphic functions
with respect to $f)$ on $\mathbb{C}^{m} .$

\subsection{Family of moving hypersurfaces}
Denote by $\mathcal{H}_{\mathbb{C}^{m}}$ the ring of all holomorphic functions on $\mathbb{C}^{m}$. Let $Q$ be a homogeneous polynomial
in $\mathcal{H}_{\mathbb{C}^{m}}\left[x_{0}, \ldots, x_{n}\right]$ of degree $d \geq 1 .$ Denote by $Q(z)$ the homogeneous polynomial over $\mathbb{C}$ obtained by substituting a specific point $z \in \mathbb{C}^{m}$ into the coefficients of $Q .$ We also call a moving hypersurface in $\mathbb{P}^{n}(\mathbb{C})$
each homogeneous polynomial $Q \in \mathcal{H}_{\mathbb{C}^{m}}\left[x_{0}, \ldots, x_{n}\right]$ such that the common zero set of all coefficients of $Q$
has codimension at least two.\par

{\bf A moving hypersurface} Let $Q$ be a moving in $\mathbb{P}^{n}(\mathbb{C})$ of degree $d(\geq 1)$ is defined by
$$Q(z)=\sum_{I\in{\mathcal{I}_{d}}}a_{I}(z)\omega^{I},$$
where $\mathcal{I}_{d}=\{(i_{0}, \ldots, i_{n})\in\mathbf{N}_{0}^{n+1};i_{0}+\cdots+ i_{n}=d\},$ $a_{I}\in\mathcal{H}_{\mathbb{C}^{m}},$ and $\omega^{I}=\omega^{i_{0}}_{0}\cdots\omega^{i_{n}}_{n}.$ We consider the  meromorphic mapping $Q':\mathbb{C}^{m}\rightarrow\mathbb{P}^{N}(\mathbb{C}),$ where $N={n+d \choose n},$ given by $$Q^{\prime}(z)=\left(a_{I_{0}}(z): \cdots: a_{I_{N}}(z)\right)\left(\mathcal{I}_{d}=\left\{I_{0}, \ldots, I_{N}\right\}\right).$$
Here $I_{0}<\cdots<I_{N}$ in the lexicographic ordering. By changing the homogeneous coordinates of $\mathbb{P}^{n}(\mathbb{C})$
if necessary, we may assume that for each given moving hypersurface as above, $a_{I_{0}} \neq 0$ (note that $I_{0}=$
$(0, \ldots, 0, d)$ and $a_{I_{0}}$ is the coefficient of $\omega_{n}^{d}$ ). We set $$\tilde{Q}=\sum_{j=0}^{N} \frac{a_{I_{j}}}{a_{I_{0}}} \omega^{I_{j}}.$$\par

The moving hypersurfaces $\mathcal{Q}$ = $\{Q_{1},\ldots,Q_{q}\}$ is said to be "slow", (with respect to $f$ ) if $\| T_{Q^{\prime}}(r)=o\left(T_{f}(r)\right)$. This is equivalent to $\| T_{\frac{a_{I_{j}}}{a_{I_{0}}}}(r)=o\left(T_{f}(r)\right)(\forall 1 \leq j \leq N),$ i.e., $\frac{a_{I_{j}}}{a_{I_{0}}} \in \mathcal{K}_{f}.$
Let $\{Q_{i}\}^{q}_{i=1}$ be a family of moving hypersurfaces in $\mathbb{P}^{n}(\mathbb{C}),$ $\deg Q_{i}=d_{i}.$ Assume that
$$Q_{i}=\sum_{I\in\mathcal{I}_{d_{i}}}a_{iI} \omega^{I}.$$\par

We denote by $\mathcal{K}_{\mathcal{Q}}$ the smallest subfield of meromorphic function field $\mathcal{M}$ which contains $\mathbb{C}$ and all $\frac{a_{i I_{s}}}{a_{i I_{t}}},$ where $a_{i I_{t}}\not\equiv 0,$ $i\in\{1,...,q\},$ $I_{t}, I_{s}\in\mathcal{I}_{d_{i}}.$ Assume that $f$ is linearly nondegenerate over $\mathcal{K}_{\mathcal{Q}}$ if there is no nonzero linear form $L\in\mathcal{K}_{\mathcal{Q}}[x_{0},\ldots,x_{n}]$ such that $L(f_{0},\ldots,f_{n})\equiv 0,$ and $f$ is algebraically nondegenerate over $\mathcal{K}_{\mathcal{Q}}$ if there is no nonzero homogeneous polynomial $Q\in\mathcal{K}_{\mathcal{Q}}[x_{0},\ldots,x_{n}]$ such that $Q(f_{0},\ldots,f_{n})\equiv 0.$\par

\subsection{ Some theorems and lemmas}
Let $f$ be a nonconstant meromorphic map of $\mathbb{C}^{m}$ into $\mathbb{P}^{n}(\mathbb{C}) .$ Denote by $\mathcal{C}_{f}$ the set of all non-negative
functions $h: \mathbb{C}^{m} \backslash A \longrightarrow[0,+\infty] \subset \overline{\mathbf{R}},$ which are of the form $$\frac{|g_{1}|+\cdots+|g_{l}|}{|g_{l+1}|+\cdots+|g_{l+k}|},$$
where $k, l \in \mathbf{N}, g_{1}, \ldots, g_{l+k} \in \mathcal{K}_{f} \backslash\{0\}$ and $A \subset \mathbb{C}^{m},$ which may depend on $g_{1}, \ldots, g_{l+k},$ is an analytic subset of codimension at least two. Then, for $h \in \mathcal{C}_{f}$ we have
$$\int_{S(r)} \log h \sigma_{m}=o\left(T_{f}(r)\right).$$

Since $\mathcal{Q}=\{Q_{1}, \ldots, Q_{q}\}$ are in $N$-subgeneral position, we have the following lemma.\par

\begin{lemma}\label{L2.1}\cite{D, Q1, Y}
For any $Q_{j_{1}},\ldots,Q_{j_{N+1}}\in\mathcal{Q},$ there exist functions $h_{1}, h_{2}\in \mathcal{C}_{f}$ such that
$$h_{2}\|f\|^{d}\leq\max_{k\in\{1,\ldots,N+1\}}|Q_{j_{k}}(f_{0},\ldots,f_{n})|\leq h_{1}\|f\|^{d}.$$
\end{lemma}

\begin{lemma} (Lemma on logarithmic derivative, see $[6]) .$
Let $f$ be a nonzero meromorphic function on $\mathrm{C}^{m}$.
$$\| \quad m(r, \frac{\mathcal{D}^{\alpha}(f)}{f})=O(\log ^{+} T(r, f))(\alpha \in \mathbf{Z}_{+}^{m}).$$
\end{lemma}
\begin{proposition}\label{P1.1}\cite{H} Let  $\Phi_{1}, \ldots, \Phi_{k}$ be meromorphic functions on $\mathbb{C}^{m}$ such that $\{\Phi_{1},$ $\ldots, \Phi_{k}\}.$ are linearly independent over $\mathbb{C} .$ Then there exists an admissible set
$$\{\alpha_{i}=(\alpha_{i 1}, \ldots, \alpha_{i m})\}_{i=1}^{k} \subset \mathbf{Z}_{+}^{m}$$
with $|\alpha_{i}|=\sum_{j=1}^{m}|\alpha_{i j}| \leq i-1(1 \leq i \leq k)$ such that the following are satisfied:\\
(i) $\{\mathcal{D}^{\alpha_{i}} \Phi_{1}, \ldots, \mathcal{D}^{\alpha_{i}} \Phi_{k}\}_{i=1}^{k}$ is linearly independent over $\mathcal{M},$ i.e., det $(\mathcal{D}^{\alpha_{i}} \Phi_{j}) \not\equiv 0,$\\
(ii) $\operatorname{det}(\mathcal{D}^{\alpha_{i}}(h \Phi_{j}))=h^{k} \cdot \operatorname{det}(\mathcal{D}^{\alpha_{i}} \Phi_{j})$ for any nonzero meromorphic function $h$ on $\mathbb{C}^{m} .$\end{proposition}

\begin{theorem}\label{TB}  (See \cite{Ru2}, Theorem 2.31) Let $f$ be a linearly nondegenerate meromorphic mapping of $\mathbb{C}^{m}$ in
$\mathbb{P}^{n}(\mathbb{C})$ with a reduced representation $\tilde{f}=(f_{0}, \ldots, f_{n})$ and let $H_{1}, \ldots, H_{q}$ be q arbitrary hyperplanes in $\mathbb{P}^{n}(\mathbb{C}).$ Then we have
$$\| \int_{S(r)} \max _{K} \log \left(\prod_{j \in K} \frac{\|\tilde{f}\| \cdot\left\|H_{j}\right\|}{\left|H_{j}(\tilde{f})\right|} \sigma_{m}\right) \leq(n+1) T_{f}(r)-N_{W^{\alpha}(f_{i})}(r)+o(T_{f}(r)),$$
where $\alpha$ is an admissible set with respect to  $\tilde{f}$ (as in Proposition  2.3) and the maximum is taken over all subsets $K \subset\{1, \ldots, q\}$ such that $\{H_{j} ; j \in K\}$ is linearly independent.
\end{theorem}


\section{Proof of Theorem \ref{T1.1}}
Firstly, we may assume that $Q_{1}, \ldots, Q_{q}$ have the same degree $\deg Q_{j}=d_{j}=d.$ Set
$$Q_{j}(z)=\sum_{I\in{\mathcal{I}_{d}}}a_{jI}(z)\textbf{x}^{I},j=1,...,q.$$
For each $j$, there exists $a_{jI_{j}} (z)$, one of the coefficients in $Q_{j}(z)$, such that $a_{jI_{j}}(z)\not\equiv 0.$ We fix this $a_{jI_{j}}$, then set $\tilde{a}_{jI}(z)=\frac{a_{jI}(z)}{a_{jI_{j}}(z)}$ and
$$\tilde{Q}_{j}(z)= \sum_{I\in\mathcal{I}_{d}}\tilde{a}_{jI}(z)\textbf{x}^{I}$$
which is a homogeneous polynomial in $\mathcal{K}_{\mathcal{Q}}[x_{0},\ldots,x_{n}]$. By definition of the proximity function and Weil function, we have
$$\lambda_{\tilde{Q}_{j}}(\textbf{f})=\log\frac{\|\textbf{f}\|^{d}\|Q_{j}\|}{|Q_{j}(\textbf{f})|}=\log\frac{\|\textbf{f}\|^{d}\|\tilde{Q}_{j}\|}{|\tilde{Q}_{j}(\textbf{f})|},$$
for $j=1,\ldots,q.$

For a fixed point $z\in \mathbb{C}^{m}\setminus \cup ^{q}_{i=1}\tilde{Q}_{i}(\tilde{f})^{-1}(\{0, \infty\}).$ We may assume that there exists a renumbering $\{1,\ldots,q\}$ such that
$$|\tilde{Q}_{1(z)}(\tilde{f})(z)|\leq|\tilde{Q}_{2(z)}(\tilde{f})(z)|\leq\cdots\leq|\tilde{Q}_{q(z)}(\tilde{f})(z)|.$$
By Lemma \ref{L2.1}, we have $\max_{j\in\{1,\ldots, N+1\}}|\tilde{Q}_{j(z)}(\tilde{f})(z)|=|\tilde{Q}_{N+1(z)}(\tilde{f})(z)|\geq h\|\tilde{f}\|^{d}$ for some $h\in \mathcal{C}_{f}, $ i. e. ,
$$\frac{\|\tilde{f}(z)\|^{d}}{|\tilde{Q}_{N+1(z)}(\tilde{f})(z)|}\leq\frac{1}{h(z)}.$$
Hence,
\begin{eqnarray}\label{E1.5}
\prod^{q}_{j=1}\frac{\|\tilde{f}(z)\|^{d}}{|\tilde{Q}_{j}(\tilde{f})(z)|}\leq\frac{1}{h^{q-N}(z)}\prod^{N}_{j=1}\frac{\|\tilde{f}(z)\|^{d}}{|\tilde{Q}_{j(z)}(\tilde{f})(z)|}.
\end{eqnarray}\par

Let $\mathcal{K}_{\mathcal{Q}}$ be an arbitrary field over $\mathbb{C}^{m}$ generated by a set of meromorphic function on $\mathbb{C}^{m}.$ Let $V$ be a sub-variety in $\mathbb{P}^{n}(\mathbb{C})$ of dimension $\ell$ defined by the homogenous ideal $I(V)\subset \mathbb{C}[x_{0},\cdots,x_{n}].$ Denote by $I_{\mathcal{K}_{\mathcal{Q}}}(V)$ the ideal in $\mathcal{K}_{\mathcal{Q}}[x_{0},\cdots,x_{n}]$ generated by $I(V).$\par

Since $f:\mathbb{C}^{m}\rightarrow V\subset\mathbb{P}^{n}(\mathbb{C})$ is algebraically nondegenerate over $\mathcal{K}_{\mathcal{Q}},$ there is no homogeneous polynomial $P\in\mathcal{K}_{\mathcal{Q}}[x_{0},\ldots,x_{n}]/I_{\mathcal{K}_{\mathcal{Q}}}(V)$ such that $P(f_{0},\ldots,f_{n})\equiv 0.$\par

For a positive integer $L$, let $\mathcal{K}_{\mathcal{Q}}[x_{0},\ldots,x_{n}]_{L}$ be the vector space of homogeneous polynomials of degree $L$, and let $I_{\mathcal{K}_{\mathcal{Q}}}(V)_{L}:=I_{\mathcal{K}_{\mathcal{Q}}}(V)\cap \mathcal{K}_{\mathcal{Q}}[x_{0},\ldots,x_{n}]_{L}$. Set $V_{L}:=\mathcal{K}_{\mathcal{Q}}[x_{0},\ldots,x_{n}]_{L}/I_{\mathcal{K}_{\mathcal{Q}}}(V)_{L}$, denote by $[g]$ the projection of $g$ in $V_{L}$. We have the following basic fact from the theory of Hilbert polynomials (e.g. see \cite{So}):\par

\begin{lemma}\label{L2.2}
 $M:= \dim_{\mathcal{K}_{\mathcal{Q}}}V_{L} =\frac{\Delta L^{\ell}}{\ell!} + \rho(L)$, where $\rho(L)$ is an $O(L^{\ell-1})$ function depending on the variety $V$. Moreover, there exists an integer $L_{0}$ such that $\rho(L)$ is a polynomial function of $L$ for $L > L_{0}$.
\end{lemma}\par

Let $a$ be an arbitrary point in $\mathbb{C}^{m}$ such that all coefficients of $P_{1},\ldots,P_{s}$ are holomorphic at
$a$, denote by $I(V(a))$ the homogeneous ideal in $\mathbb{C}[x_{0},\ldots, x_{n}]$ generated by $P_{1}(a),\ldots,P_{s}(a)$, let $V(a)$ be the variety in $V(a)\subset\mathbb{P}^{n}(\mathbb{C})$ defined by $I(V(a))$, then we have
\begin{lemma}\label{L2.7.}\cite{D, Q1, Y}
 $\dim V(a)=\ell,$ for all $a\in\mathbb{C}$ excluding a discrete subset.\par
\end{lemma}

Next, we prove the following lemma concerning on the hypersurfaces located in $N$-subgeneral position with index $\kappa,$ which plays the role in this paper. The method of it is originally from Quang \cite{Q1, Q2}.\par

\begin{lemma}\label{L2.4}
Let $\tilde{Q}_{1},\ldots,\tilde{Q}_{N+1}$ be homogeneous polynomials in $\mathcal{K}_{\mathcal{Q}}[x_{0},\ldots,x_{n}]$ of the same degree $d$ $\geq$ 1, in (weakly) $N$-subgeneral position with index $\kappa$ in $V.$ For each point $a$ $\in$ $\mathbb{C}^{m}$ satisfying the following conditions:\par

(i) the coefficients of $\tilde{Q}_{1},\ldots,\tilde{Q}_{N+1}$ are holomorphic at $a$,\par

(ii) $\tilde{Q}_{1}(a),\ldots,\tilde{Q}_{N+1}(a)$ have no non-trivial common zeros,\par

(iii) $\dim V(a)=\ell,$\par

\noindent then there exist homogeneous polynomials
$\tilde{P}_{1}(a)=\tilde{Q}_{1}(a),\ldots,\tilde{P}_{\kappa}(a)=\tilde{Q}_{\kappa}(a),$ $\tilde{P}_{\kappa+1}(a),\ldots, \tilde{P}_{\ell+1}(a)$ $\in$ $\mathbb{C}[x_{0},\ldots,x_{n}]$ with
$$\tilde{P}_{t}(a)=\sum^{N-\ell+t}_{j=\kappa+1}c_{tj}\tilde{Q}_{j}(a),\,\,\, c_{tj}\in\mathbb{C},\,\,\, t\geq \kappa+1,$$
such that
$$\left(\bigcap^{\ell+1}_{t=1}\tilde{P}_{t}(a)\right)\cap V =\emptyset.$$
\end{lemma}

\begin{proof} We assume that $\tilde{Q}_{i}$ $(1\leq i\leq N+1)$ has the following form
$$\tilde{Q}_{i}=\sum_{I\in\tau_{d}}a_{iI}\omega^{I}.$$
By the definition of the $N$-subgeneral position, there exists a point $a\in\mathbb{C}$ such that the following system of equations
$$\tilde{Q}_{i}(a)(\omega_{0},\ldots,\omega_{n})=0,\,\,\, 1\leq i\leq N+1$$
has only trivial solution $(0,\ldots,0).$ We may assume that $\tilde{Q}_{i}(a)\not\equiv0$ for all $1\leq i\leq N+1.$ \par

For each homogeneous polynomial $\tilde{Q}\in\mathbb{C}[x_{0},\ldots,x_{n}],$ we denote by $D$ the fixed hypersurface in $\mathbb{P}^{n}(\mathbb{C})$ defined by $\tilde{Q}$, i.e.,
$$D=\{(\omega_{0},\ldots,\omega_{n})\in\mathbb{P}^{n}(\mathbb{C})\mid\tilde{Q}(\omega_{0},\ldots,\omega_{n})=0\}.$$
Setting $\tilde{P}_{1}(a)=\tilde{Q}_{1}(a),\ldots,\tilde{P}_{\kappa}(a)=\tilde{Q}_{\kappa}(a),$ we see that
\begin{eqnarray*}
\dim\left(\bigcap^{t}_{i=1}D_{i}(a)\cap V(a)\right)\leq N-t,\,\,\, t=N-\ell+\kappa+1,\ldots, N+1,
\end{eqnarray*}
where $\dim \emptyset$ = -$\infty$.\par

And for any $\kappa$ moving hypersurfaces, we have
$$\dim\left(\bigcap^{\kappa}_{j=1}D_{j}(a)\cap V(a)\right) = \ell-\kappa$$
(On the one hand, it is obvious $\dim\left(\bigcap^{\kappa}_{j=1}D_{j}(a)\cap V\right)\geq \ell-\kappa;$ on the other hand,  according to the definition of $N$-subgeneral position with index $\kappa$, we have $\dim\left(\bigcap^{\kappa}_{j=1}D_{j}(a)\cap V\right)\leq \ell-\kappa$).\par

Step 1. We will construct $\tilde{P}_{\kappa+1}$ as follows. For each irreducible componet $\Gamma$ of dimension $\ell-\kappa$ of $(\bigcap_{i=1}^{\kappa}\tilde{P}_{i}(a)\cap V(a)),$ we put
\begin{eqnarray*}
V_{1\Gamma}&=&\{c=(c_{\kappa+1},\ldots,c_{N-\ell+\kappa+1})\in\mathbb{C}^{N-\ell+1},\Gamma\subset D_{c}(a),\,\,\,\\&& \mbox{where}\,\,\,\tilde{Q}_{c} = \sum^{N-\ell+\kappa+1}_{j=\kappa+1}c_{j}\tilde{Q}_{j}\}.
\end{eqnarray*}
By definition, $V_{1\Gamma}$ is a subspace of $\mathbb{C}^{N-\ell+1}$. Since
$$\dim\left(\bigcap^{N-\ell+\kappa+1}_{i=1}D_{i}(a)\cap V(a)\right)\leq \ell-\kappa-1,$$
there exists $i\in\{\kappa+1,\ldots,N-\ell+\kappa+1\}$ such that $\Gamma\not\subset D_{i}(a).$ This implies that $V_{1\Gamma}$ is a proper subspace of $\mathbb{C}^{N-\ell+1}.$ Since the set of irreducible components of dimension $\ell-\kappa$ of $(\bigcap_{i=1}^{\kappa}\tilde{P}_{i}(a)\cap V(a))$ is at most countable, we have
$$\mathbb{C}^{N-\ell+\kappa}\setminus\bigcup_{\Gamma}V_{1\Gamma}\neq\emptyset.$$
Hence, there exists $(c_{1(\kappa+1)},\ldots,c_{1(N-\ell+\kappa+1)})\in\mathbb{C}^{N-\ell+1}$ such that
$\Gamma\not\subset\tilde{P}_{\kappa+1}(a),$ where $\tilde{P}_{\kappa+1}=\sum^{N-\ell+\kappa+1}_{j=\kappa+1}c_{1j}\tilde{Q}_{j},$ for all irreducible components of dimension $\ell-\kappa$ of $(\bigcap_{i=1}^{\kappa}\tilde{P}_{i}(a)\cap V(a)).$ This clearly implies that $$\dim \left(\bigcap_{i=1}^{\kappa+1}\tilde{P}_{i}(a)\cap V(a)\right)\leq \ell-(\kappa+1).$$\par

Step 2. We will construct $P_{\kappa+2}$ as follows. For each irreducible componet $\Gamma'$ of dimension $\ell-\kappa-1$ of $\left(\bigcap_{i=1}^{\kappa+1}\tilde{P}_{i}(a)\cap V(a)\right),$ we put
\begin{eqnarray*}
V_{2\Gamma'}&=&\{c=(c_{\kappa+1},\ldots,c_{N-\ell+\kappa+2})\in\mathbb{C}^{N-\ell+2},\Gamma'\subset D_{c}(a),\,\,\,\\&&where\,\,\,\tilde{Q}_{c} = \sum^{N-\ell+\kappa+2}_{j=\kappa+1}c_{j}\tilde{Q}_{j}\}.
\end{eqnarray*}
Then $V_{2\Gamma'}$ is a subspace of $\mathbb{C}^{N-\ell+2}$. Since dim$\left(\bigcap^{N-\ell+\kappa+2}_{i=1}D_{i}(a)\cap V(a)\right)\leq \ell-\kappa-2$, there exists $i\in\{\kappa+1,\ldots,N-\ell+\kappa+2\}$ such that $\Gamma'\not\subset D_{i}(a).$ This implies that $V_{2\Gamma'}$ is a proper subspace of $\mathbb{C}^{N-\ell+2}.$ Since the set of irreducible components of dimension $\ell-\kappa-1$ of $\left(\bigcap_{i=1}^{\kappa+1}\tilde{P}_{i}(a)\cap V(a)\right)$ is at most countable, we also have
$$\mathbb{C}^{N-\ell+\kappa+1}\setminus\bigcup_{\Gamma'}V_{2\Gamma'}\neq\emptyset.$$
Hence, there exists $(c_{2(\kappa+1)},\ldots,c_{2(N-\ell+\kappa+2)})\in\mathbb{C}^{N-\ell+2}$ such that $\Gamma'\not\subset \tilde{P}_{\kappa+2}(a),$
where $\tilde{P}_{\kappa+2}=\sum^{N-\ell+\kappa+2}_{j=\kappa+1}c_{2j}\tilde{Q}_{j},$ for all irreducible components of dimension $\ell-\kappa-1$ of $\left(\bigcap_{i=1}^{\kappa+1}\tilde{P}_{i}(a)\cap V(a)\right).$ This clearly implies that $$\dim \left(\bigcap_{i=1}^{\kappa+2}\tilde{P}_{i}(a)\cap V(a)\right)\leq \ell-(\kappa+2).$$\par

Repeat again the above steps, after $(\ell+1-\kappa)$-th step we get the hypersurface $\tilde{P}_{\kappa+1}(a),\ldots,\tilde{P}_{\ell+1}(a)$ satisfying that
$$\dim\left(\bigcap^{t}_{j=1}\tilde{P}_{j}(a)\cap V(a)\right)\leq \ell-t,$$ where $t=\kappa+1,\ldots, \ell+1.$\par

In particular, $\left(\bigcap^{\ell+1}_{j=1}\tilde{P}_{j}(a)\cap V(a)\right)=\emptyset.$ This yields that $\tilde{P}_{1}(a),\ldots,\tilde{P}_{\ell+1}(a)$ are in general position. We complete the proof of the lemma.\end{proof}\par

Since there are only finitely many choice of $N+1$ polynomials from ${\tilde{Q}1,\ldots,\tilde{Q}_{q}}$, the total number of such $\tilde{P}_{j}'s$ is finite, so there exists a constant $C>0,$ for $t=\kappa+1,\ldots,\ell$ and all $z\in\mathbb{C}^m$ (excluding all zeros and poles of all $\tilde{Q}_{j}(f))$, by Lemma \ref{L2.4} we can construct $\tilde{P}_{1}=\tilde{Q}_{1},$ $\ldots,$ $\tilde{P}_{\kappa}=\tilde{Q}_{\kappa},$ $\tilde{P}_{\kappa+1},$ $\ldots,$ $\tilde{P}_{\ell+1}$ from ${\tilde{Q}_{1},\ldots,\tilde{Q}_{N+1}}$ such that
\begin{equation*}
\mid\tilde{P}_{t(z)}(\tilde{f})(z)\mid\leq C\max_{\kappa+1\leq j\leq N-\ell+t}\mid\tilde{Q}_{j(z)}(\tilde{f})(z)\mid=C\mid\tilde{Q}_{(N-\ell+t)(z)}(\tilde{f})(z)\mid
\end{equation*}
for $\kappa+1\leq t\leq \ell,$ and thus,
\begin{equation*}
\lambda_{\tilde{Q}_{(N-\ell+t)(z)}(\tilde{f})(z)} \leq \lambda_{\tilde{P}_{t(z)}(\tilde{f})(z)} + \log h'', h''\in\mathcal{C}_{f},\,\,\,\mbox{for}\,\,\, \kappa+1\leq t\leq \ell.
\end{equation*}
Combing the above inequality with (\ref{E1.5}), we have
\begin{eqnarray*}
&&\sum^{q}_{j=1}\lambda_{\tilde{Q}_{j(z)}(\tilde{f})(z)}\\\nonumber&\leq&\sum^{\kappa}_{j=1}\lambda_{\tilde{Q}_{j(z)}(\tilde{f})(z)}
+\sum^{N-\ell+\kappa}_{j=\kappa+1}\lambda_{\tilde{Q}_{j(z)}(\tilde{f})(z)}+\sum_{N-\ell+\kappa+1}^{N}\lambda_{\tilde{Q}_{j(z)}(\tilde{f})(z)}+\log h'\\\nonumber
&\leq&\sum^{\kappa}_{j=1}\lambda_{\tilde{Q}_{j(z)}(\tilde{f})(z)}+\sum^{N-\ell+\kappa}_{j=\kappa+1}\lambda_{\tilde{Q}_{j(z)}(\tilde{f})(z)}+ \sum^{\ell}_{j=\kappa+1}\lambda_{\tilde{P}_{j(z)}(\tilde{f})(z)}+\log h'',\\
\nonumber&=&\sum^{\ell}_{j=1}\lambda_{\tilde{P}_{j(z)}(\tilde{f})(z)}+\sum^{N-\ell+\kappa}_{j=\kappa+1}\lambda_{\tilde{Q}_{j(z)}(\tilde{f})(z)}+\log h''.\end{eqnarray*}
This gives that
if $N-\ell\leq \kappa,$ we have
\begin{eqnarray} \label{E2.11}\sum^{q}_{j=1}\lambda_{\tilde{Q}_{j(z)}(\tilde{f})(z)}
&\leq&\sum^{\ell}_{j=1}\lambda_{\tilde{P}_{j(z)}(\tilde{f})(z)}+\sum_{j=1}^{N-\ell}\lambda_{\tilde{Q}_{j(z)}(\tilde{f})(z)}+\log h''\\\nonumber
&=&\sum^{\ell}_{j=1}\lambda_{\tilde{P}_{j(z)}(\tilde{f})(z)}+\sum^{N-\ell}_{j=1}\lambda_{\tilde{P}_{j(z)}(\tilde{f})(z)}+\log h'',
\end{eqnarray}
and if $N-\ell\geq \kappa$, we get \begin{eqnarray}\label{E2.12}\sum^{q}_{j=1}\lambda_{\tilde{Q}_{j(z)}(\tilde{f})(z)}
&\leq&\sum^{\ell}_{j=1}\lambda_{\tilde{P}_{j(z)}(\tilde{f})(z)}+\sum_{j=1}^{N-\ell}\lambda_{\tilde{Q}_{j(z)}(\tilde{f})(z)}+\log h''\\\nonumber&\leq&\sum^{\ell}_{j=1}\lambda_{\tilde{P}_{j(z)}(\tilde{f})(z)}+\frac{N-\ell}{\kappa}\sum^{\kappa}_{j=1}\lambda_{\tilde{Q}_{j(z)}(\tilde{f})(z)}+\log h'''
\\\nonumber&=&\sum^{\ell}_{j=1}\lambda_{\tilde{P}_{j(z)}(\tilde{f})(z)}+\frac{N-\ell}{\kappa}\sum^{\kappa}_{j=1}\lambda_{\tilde{P}_{j(z)}(\tilde{f})(z)}+\log h''',
\end{eqnarray}where
$h'''\in\mathcal{C}_{f}.$ Hence, by (\ref{E2.11}) and (\ref{E2.12}), we get
\begin{eqnarray}\label{E2.13}
&&\sum^{q}_{j=1}\lambda_{\tilde{Q}_{j(z)}(\tilde{f})(z)}\\\nonumber&\leq&\sum^{\ell}_{j=1}\lambda_{\tilde{P}_{j(z)}(\tilde{f})(z)}+\frac{N-\ell}{\max\{1,\min\{N-\ell,\kappa\}\}}\sum^{\ell}_{j=1}\lambda_{\tilde{P}_{j(z)}(\tilde{f})(z)}+\log h^*\\\nonumber&=&\left(1+\frac{N-\ell}{\max\{1,\min\{N-\ell,\kappa\}\}}\right)\sum^{\ell}_{j=1}\lambda_{\tilde{P}_{j(z)}(\tilde{f})(z)}+\log h^*,
\end{eqnarray}where $h^*=\max\{h'', h'''\}\in\mathcal{C}_{f}.$\par

Fix a basis $\{[\phi_{1}],\ldots,[\phi_{M}]\}$ of $V_{L}$ with $[\phi_{1}],\ldots,[\phi_{M}]\in \mathcal{K}_{\mathcal{Q}}[x_{0},\ldots,x_{n}],$ and let
$$F=[\phi_{1}(\tilde{f}),\ldots,\phi_{M}(\tilde{f})]:\mathbb{C}\rightarrow\mathbb{P}^{M-1}(\mathbb{C}).$$
Since $\tilde{f}$ satisfies $P(\tilde{f})\not\equiv0$ for all homogeneous polynomials $P\in \mathcal{K}_{\mathcal{Q}}[x_{0},\ldots,x_{n}]/$ $I_{\mathcal{K} _{\mathcal{Q}}}(V) $, $F$ is linearly nondegenerate over $\mathcal{K}_{\mathcal{Q}}.$ We have
\begin{eqnarray}\label{E3.1}
T_{F}(r)=LT_{f}(r)+o(T_{f}(r)).
\end{eqnarray}\par
For every positive integer $L$ divided by $d,$ we use the following filtration of the vector space $V_{L}$ with respect to $\tilde{P}_{1(z)},\ldots,\tilde{P}_{\ell(z)}.$ This is a generalization of Corvaja-Zannier¡¯s filtration \cite{Co}, see in the two references \cite{Ru1, C, Y}.\par

Arrange, by the lexicographic order, the $\ell$-tuples $\mathbf{i}= (i_{1},\ldots,i_{\ell})$ of non-negative integers and set $\|\mathbf{i}\|=\sum_{j}i_{j}.$\par

\begin{definition}\cite{Ru1, C, Y}(i) For each $\mathbf{i}\in \mathbf{Z}^{\ell}_{\geq0}$ and non-negative integer $L$ with $L\geq d\|\mathbf{i}\|,$ denote by $I^{\mathbf{i}}_{L}$ the subspace of $\mathcal{K}_{\mathcal{Q}}[x_{0},\ldots,x_{n}]_{L-d\|\mathbf{i}\|}$ consisting of all $\gamma\in \mathcal{K}_{\mathcal{Q}}[x_{0},\ldots,x_{n}]_{L-d\|\mathbf{i}\|}$ such that
$$\tilde{P}^{i_{1}}_{1(z)}\cdots\tilde{P}^{i_{\ell}}_{\ell(z)}\gamma-\sum_{\textbf{e}=(e_{1},\ldots,e_{\ell})>\mathbf{i}}\tilde{P}^{e_{1}}_{1(z)}\cdots\tilde{P}^{e_{\ell}}_{\ell(z)}\gamma_{e}\in I_{\mathcal{K}_{\mathcal{Q}}}(V)_{L}$$
\,\,\, (or$[\tilde{P}^{i_{1}}_{1(z)}\cdots\tilde{P}^{i_{\ell}}_{\ell(z)}\gamma]=[\sum_{\textbf{e}=(e_{1},\ldots,e_{\ell})>\mathbf{i}}\tilde{P}^{e_{1}}_{1(z)}\cdots\tilde{P}^{e_{\ell}}_{\ell(z)}\gamma_{e}]$ on $V_{L}$)\\ for some $\gamma_{e}\in$ $\mathcal{K}_{\mathcal{Q}}[x_{0}, \ldots, x_{n}]_{L-d\|\textbf{e}\|}.$\\
(ii) Denote by $I^{\mathbf{i}}$ the homogeneous ideal in $\mathcal{K}_{\mathcal{Q}}[x_{0},\ldots,x_{n}]$ generated by $\bigcup_{L\geq d\|\mathbf{i}\|}I^{\mathbf{i}}_{L}.$
\end{definition}

\begin{remark}
\label{R2.6} \cite{Ru1, C, Y} From this definition, we have the following properties.\par

(i). $(I_{\mathcal{K}_{\mathcal{Q}}}(V),\tilde{P}_{1(z)}\cdots\tilde{P}_{\ell(z)})_{L-d\|\mathbf{i}\|}\subset I^{\mathbf{i}}_{L}\subset \mathcal{K}_{\mathcal{Q}}[x_{0},\ldots,x_{n}]_{L-d\|\mathbf{i}\|},$ where we denote by
$(I_{\mathcal{K}_{\mathcal{Q}}}(V),\tilde{P}_{1(z)}\cdots\tilde{P}_{\ell(z)})$ the ideal in $\mathcal{K}_{\mathcal{Q}}[x_{0},\ldots,x_{n}]$ generated by $I_{\mathcal{K}_{\mathcal{Q}}}(V)\cup\{{\tilde{P}_{1(z)}\cdots\tilde{P}_{\ell(z)}}\}.$\par

(ii). $I^{\mathbf{i}}\cap \mathcal{K}_{\mathcal{Q}}[x_{0},\ldots,x_{n}]_{L-d\|\mathbf{i}\|}=I^{\mathbf{i}}_{L}.$\par

(iii). $\frac{\mathcal{K}_{\mathcal{Q}}[x_{0},\ldots,x_{n}]}{I^{\mathbf{i}}}$ is a graded module over $\mathcal{K}_{\mathcal{Q}}[x_{0},\ldots,x_{n}].$\par

(iv). If $\mathbf{i}_{1}-\mathbf{i}_{2}:=(i_{1,1}-i_{2,1},\ldots,i_{1,\ell}-i_{2,\ell})\in\mathbf{Z}^{\ell}_{\geq0},$ then $I^{\mathbf{i}_{2}}_{L}\subset I^{\mathbf{i}_{1}}_{L+d\|\mathbf{i}_{1}\|-d\|\mathbf{i}_{2}\|}.$ Hence $I^{\mathbf{i}_{2}}\subset I^{\mathbf{i}_{1}}.$
\end{remark}

\begin{lemma}\label{L2.9.}\cite{Ru1, C, Y}
$\{I^{\mathbf{i}}|\mathbf{i}\in\mathbf{Z}^{\ell}_{\geq0}\}$ is finite set.
\end{lemma}

Denote by
\begin{eqnarray}\label{E1.7}
&&\Delta^{\mathbf{i}}_{L}:=dim_{\mathcal{K}_{\mathcal{Q}}}\frac{\mathcal{K}_{\mathcal{Q}}[x_{0},\ldots,x_{n}]_{L-d\|\mathbf{i}\|}}{I^{\mathbf{i}}_{L}}.
\end{eqnarray}

\begin{lemma}\label{L2.10.}\cite{Ru1, C, Y}
(i). There exists a positive integer  $L_{0}$ such that, for each  $\mathbf{i}\in\mathbf{Z}^{\ell}_{\geq0},$ $\Delta^{\mathbf{i}}_{L}$ is independent of  $L$  for all  $L$  satisfying  $L-d\|\mathbf{i}\|>L_{0}.$\par

(ii). There is an integer  $\bar{\Delta}$  such that  $\Delta^{\mathbf{i}}_{L}\leq\bar{\Delta}$ for all  $\mathbf{i}\in\mathbf{Z}^{\ell}_{\geq0}$  and  $L$  satisfying  $L-d\|\mathbf{i}\|>0.$
\end{lemma}

Set $\Delta_{0}:=\min_{\mathbf{i}\in\mathbf{Z}^{\ell}_{\geq0}}\Delta^{\mathbf{i}}=\Delta^{\mathbf{i}_{0}}$ for some $\mathbf{i}_{0}\in\mathbf{Z}^{\ell}_{\geq0}.$\par

\begin{remark}\label{R2.10.}\cite{Ru1, C, Y}
By  (iv) of  Remark \ref{R2.6}, if $\mathbf{i}-\mathbf{i}_{0}\in\mathbf{Z}^{\ell}_{\geq0},$ then $\Delta^{\mathbf{i}}\leq\Delta^{\mathbf{i}_{0}}.$
\end{remark}

Now, for an integer $L$ big enough, divisible by $d$, we construct the following filtration of $V_{L}$ with respect to $\{{\tilde{P}_{1(z)}\cdots\tilde{P}_{\ell(z)}}\}.$
Denote by $\tau_{L}$ the set of $\mathbf{i}\in\mathbf{Z}^{\ell}_{\geq0},$ with $L-d\|\mathbf{i}\|>0.$ arranged by the lexicographic order.
Define the spaces $W_{\mathbf{i}}=W_{L,\mathbf{i}}$ by
$$W_{\mathbf{i}}=\sum_{\textbf{e}>\mathbf{i}}\tilde{P}^{e_{1}}_{1(z)}\cdots\tilde{P}^{e_{\ell}}_{\ell(z)}\cdot \mathcal{K}_{\mathcal{Q}}[x_{0},\ldots,x_{n}] _{L-d\|\textbf{e}\|}.$$
Plainly $W_{(0,\ldots,0)}=\mathcal{K}_{\mathcal{Q}}[x_{0},\ldots,x_{n}]_{L}$ and $W_{\mathbf{i}}\supset W_{\mathbf{i}'}$ if $\mathbf{i}'>\mathbf{i},$ so $\{{W_{\mathbf{i}}}\}$ is a filtration of $\mathcal{K}_{\mathcal{Q}}[x_{0},\ldots,x_{n}]_{L}.$ Set $W^{\ast}_{\mathbf{i}}=\{[g]\in V_{L}|g\in W_{\mathbf{i}}\}.$ Hence,$\{{W^{\ast}_{\mathbf{i}}}\}$ is a filtration of $V_{L}.$\par

\begin{lemma}\label{L2.8}\cite{Ru1, C, Y}
Suppose that $\mathbf{i}'$ follows $\mathbf{i}$ in the lexicographic order, then
$$\frac{W^{\ast}_{\mathbf{i}}}{W^{\ast}_{\mathbf{i}'}}\cong\frac{\mathcal{K}_{\mathcal{Q}}[x_{0},\ldots,x_{n}]_{L-d\|\mathbf{i}\|}}{I^{\mathbf{i}}_{L}}.$$
\end{lemma}
Combining with the notation (\ref{E1.7}), we have
$$\dim\frac{W^{\ast}_{\mathbf{i}}}{W^{\ast}_{\mathbf{i}'}}=\Delta^{\mathbf{i}}_{L}.$$\\
Set
$$\tau^{0}_{L}=\{\mathbf{i}\in\tau_{L}|L-d\|\mathbf{i}\|>L_{0}. and \mathbf{i}-\mathbf{i}_{0}\in\mathbf{Z}^{\ell}_{\geq0}\}.$$
We have the following properties.
\begin{lemma}\label{L2.17.}\cite{Ru1, C, Y}
(i). $\Delta_{0}=\Delta^{\mathbf{i}}$ for all $\mathbf{i}\in\tau^{0}_{L}.$\par

(ii). $\sharp\tau^{0}_{L}=\frac{1}{d^{\ell}}\frac{L^{\ell}}{\ell!}+O(L^{\ell}-1).$\par

(iii). $\Delta^{\mathbf{i}}_{L}=\Delta d^{\ell}$ for all $\mathbf{i}\in\tau^{0}_{L}.$
\end{lemma}

We can choose a basis $\mathcal{B}=\{[\psi_{1}],\ldots,[\psi_{M}]\}$ of $V_{L}$ with respect to the above filtration. Let $[\psi_{s}]$ be an element of the basis, which lies in $W^{\ast}_{\mathbf{i}}/W^{\ast}_{\mathbf{i}'},$ we may write $\psi_{s}=\tilde{P}^{e_{1}}_{1(z)}\cdots\tilde{P}^{e_{\ell}}_{\ell(z)}\gamma,$ where $\gamma\in \mathcal{K}_{\mathcal{Q}}[x_{0},\ldots,x_{n}]_{L-d\|\mathbf{i}\|}.$ For every $1\leq j\leq \ell,$ we have
\begin{eqnarray}\label{E1.10}
&&\sum_{\mathbf{i}\in\tau_{L}}\Delta^{\mathbf{i}}_{L}i_{j}=\frac{\Delta L^{\ell+1}}{(\ell+1)!d}+O(L^{\ell})
\end{eqnarray}
(The proof of \eqref{E1.10} can be seen in \cite[the equality (3.6)]{Ru1}). Hence, by \eqref{E1.10} and the definition of the Weil function, we obtain
\begin{eqnarray}\label{E1.14}
\,\,\,\,\,\,\,\,\,\,\,\,\,\,\,\,\,\,\,\sum^{M}_{s=1}\lambda_{\psi_{s}}(\tilde{f}(z))\geq\left(\frac{\Delta L^{\ell+1}}{(\ell+1)!d}+O(L^{\ell})\right)\cdot\sum^{\ell}_{j=1}\lambda_{\tilde{P}_{j(z)}}(\tilde{f}(z))+\log h^{**},
\end{eqnarray} where $h^{**}\in \mathcal{C}_{f}.$ The basis $[\psi_{1}],\ldots,[\psi_{M}]$ can be written as linear forms $L_{1},\ldots,L_{M}$ (over $\mathcal{K}_{\mathcal{Q}}$) in the basis $[\phi_{1}],\ldots, [\phi_{M}]$ and $\psi_{s}(\tilde{f})=L_{s}(\tilde{f}).$ Since there are only finitely many choices of ${\tilde{Q}_{1(z)},\ldots,\tilde{Q}_{(N+1)(z)}},$ the collection of all possible linear forms $L_{s}(1\leq s\leq M)$ is a finite set, and denote it by $\mathcal{L}:=\{L\mu\}^{\Lambda}_{\mu=1}$ $(\Lambda<+\infty).$ It is easy to see that $\mathcal{K}_{\mathcal{L}}\subset\mathcal{K}_{\mathcal{Q}}.$
\begin{lemma}\label{L2.11.} (Product to the sum estimate, see \cite{Ru2}) Let $H_{1},\ldots,H_{q}$ be hyperplanes in $\mathbb{P}^{n}(\mathbb{C})$ located in general position. Denote by $T$ the set of all injective maps $\mu:\{0,1,\ldots,n\}\rightarrow\{1,\ldots,q\}.$ Then
$$\sum^{q}_{j=1}m_{f}(r,H_{j})\leq\int^{2\pi}_{0}\max_{\mu\in T}\sum^{n}_{i=0}\lambda_{\tilde{H}_{\mu(i)}}(f(re^{i\theta}))\frac{d\theta}{2\pi}+O(1)$$ holds for all $r$ outside a set with finite Lebesgue measure.
\end{lemma}
By \eqref{E2.13}, \eqref{E1.14} and Lemma \ref{L2.11.}, take integration on the sphere of radius $r$, we have
\begin{eqnarray}\label{E1.16}
&&\frac{\Delta L^{\ell+1}}{(\ell+1)!d}(1+O(1))\cdot\sum^{q}_{j=1}m_{f}(r,\tilde{Q}_{j})
\leq\left(\frac{N-\ell}{\max\{1,\min\{N-\ell,\kappa\}\}}+1\right)\\\nonumber&&\cdot\int^{2\pi}_{0}\max_{\mathcal{K}}\sum_{j\in\mathcal{K}}\lambda_{L_{j}}(\tilde{f}(re^{i\theta}))\frac{d\theta}{2\pi}+o(T_{f}(r))
\end{eqnarray} for all $r$ outside a set with finite Lebesgue measure,
where the set $\mathcal{K}$ ranges over all subset of $\{1,\ldots,\Lambda\}$ such that the linear forms $\{L_{j}\}_{j\in \mathcal{K}}$ are linearly independent.\par

By Theorem \ref{TB}, we have, for any $\epsilon>0$,
\begin{eqnarray}\label{E1.17}
&&\int^{2\pi}_{0}\max_{\mathcal{K}}\sum_{j\in\mathcal{K}}\lambda_{L_{j}}(\tilde{f}(re^{i\theta}))\frac{d\theta}{2\pi}\leq(M+\varepsilon)T_{F}(r)-N_{W}(r,0)+o(T_{f}(r))
\end{eqnarray}holds for all $r$ outside a set $E$ with finite Lebesgue measure.
Taking $\varepsilon=\frac{1}{2}$ in (\ref{E1.17}) and (\ref{E1.16}), and (\ref{E3.1}), we obtain
\begin{eqnarray*}
&&\frac{\Delta L^{\ell+1}}{(\ell+1)!d}(1+O(1))\cdot\sum^{q}_{j=1}m_{f}(r,\tilde{Q}_{j})\\
&&\leq\left(\frac{N-\ell}{\max\{1,\min\{N-\ell,\kappa\}\}}+1\right)\left((M+\epsilon)T_{F}(r)-N_{W}(r,0)+o(T_{f}(r))\right)\\
&&\leq\left(\frac{N-\ell}{\max\{1,\min\{N-\ell,\kappa\}\}}+1\right)(\frac{\Delta L^{\ell}}{\ell!} + \rho(L)+\epsilon)LT_{f}(r)+o(T_{f}(r)),
\end{eqnarray*}holds for all $r\not\in E,$ where $W$ is the Wronskian of $F_1,\ldots,F_M.$\par

Take $L$ large enough such that $\epsilon<\left(\frac{N-\ell}{\max\{1,\min\{N-\ell,\kappa\}\}}+1\right)o(1),$ where $\epsilon>0$ is any given in the theorem. Then we have\\
\begin{eqnarray}\label{E99}&&\sum^{q}_{j=1}\frac{1}{d}m_{f}(r,\tilde{Q}_{j})\leq\left(\left(\frac{N-\ell}{\max\{1,\min\{N-\ell,\kappa\}\}}+1\right)(\ell+1)+\epsilon\right)T_{f}(r)\end{eqnarray}
holds for all $r\not\in E.$ \par

And by the first main theorem, (\ref{E99}) can be writen
\begin{eqnarray*}&&\left(q-\left(\frac{N-\ell}{\max\{1,\min\{N-\ell,\kappa\}\}}+1\right)(\ell+1)-\epsilon\right)T_{f}(r)\leq\sum^{q}_{j=1}\frac{1}{d}N_{f}(r,\tilde{Q}_{j}).\end{eqnarray*}

Secondly, for the general case whenever all $Q_{j}$ may not have the same degree. Then consider $Q_{j}^{\frac{d}{d_{j}}}$ instead of $Q_{j}.$ We have
$N_{f}(r, Q_{j})=\frac{d_{j}}{d}N_{f}(r, Q_{j}^{\frac{d}{d_{j}}}).$ Then the theorem is proved immediately.

\section{Proof of Theorem \ref{T1.2}}
Replacing $Q_{i}$ by $Q_{i}^{d/d_{i}}$ if necessary with the note that $$\frac{1}{d}N^{[L_{0}]}(r,f*Q_{i}^{d/d_{i}})\leq\frac{1}{d_{i}}N^{[L_{j}]}(r,f*Q_{i}),$$
we may assume that all hypersurfaces $Q_{i}(1\leq i\leq q)$ are of the same degree $d.$ We may also assume that $q>(\frac{N-n}{\max\{1,\min\{N-n,\kappa\}\}}+1)(n+1).$\par

Consider a reduced representation $\tilde{f}=(f_{0},\ldots,f_{n}):\mathbb{C}\rightarrow\mathbb{C}^{n+1}$ of $f.$ We also note that
$$N^{[L_{0}]}_{Q_{i}(\tilde{f})}(r)=N^{[L_{0}]}_{\tilde{Q}_{i}(\tilde{f})}(r)+o(T_{f}(r)).$$
Then without loss of generality we may assume that $Q_{i}\in \mathcal{K}_{f}[x_{0}, \ldots, x_{n}].$

We set $$\mathcal{I}=\{(i_{1},\ldots,i_{N+1});1\leq i_{j}\leq q, i_{j}\neq i_{t}\,\,\, \forall j\neq t\}.$$
For each $I=(i_{1},\ldots,i_{N+1})\in\mathcal{I},$ we denote by $P_{I1},\ldots,P_{I(n+1)}$ the moving hypersurfaces obtained in Lemma \ref{L2.4} with respect to the family of moving hypersurfaces $\{Q_{i_{1}},\ldots,Q_{i_{N+1}}\}.$ It is easy to see that there exists a positive function $h\in\mathcal{C}_{f}$ such that
$$|P_{It}(\omega)|\leq h\max_{\kappa+1\leq j\leq N+1-n+t}|Q_{i_{j}}(\omega)|,\,\,\,\kappa+1\leq t\leq n,$$ for all $I\in\mathcal{I}$ and $\omega=(\omega_0,\ldots,\omega_n)\in \mathbb{C}^{n+1}.$ \par

For a fixed point $z\in \mathbb{C}^{m}\setminus \cup ^{q}_{i=1}Q_{i}(\tilde{f})^{-1}(\{0, \infty\}).$ We may assume that such that $$|Q_{i_{1}}(\tilde{f})(z)|\leq|Q_{i_{2}}(\tilde{f})(z)|\leq\cdots\leq|Q_{i_{q}}(\tilde{f})(z)|.$$
Let $I=(i_{1},\ldots,i_{N+1}).$ Since $P_{I1},\ldots,P_{I(n+1)}$ are in weakly general position, there exist functions $g_{0},g\in\mathcal{C}_{f},$ which may be chosen independent of $I$ and $z,$ such that
$$\|\tilde{f}(z)\|^{d} \leq g_{0 }(z)\max_{1\leq j\leq n+1}|P_{Ij}(\tilde{f})(z)|\leq g(z)|Q_{i_{N+1}}(\tilde{f})(z)|.$$
Therefore, we have
\begin{eqnarray}
&&\prod^{q}_{i=1}\frac{\|\tilde{f}(z)\|^{d}}{|Q_{i}(\tilde{f})(z)|}\leq g^{q-N}(z)\prod^{N}_{j=1}\frac{\|\tilde{f}(z)\|^{d}}{|Q_{i_{j}}(\tilde{f})(z)|}\\
\nonumber&=&g^{q-N}(z)\prod^{\kappa}_{j=1}\frac{\|\tilde{f}(z)\|^{d}}{|Q_{i_{j}}(\tilde{f})(z)|}\cdot\prod^{N-n+\kappa}_{j=\kappa+1}
\frac{\|\tilde{f}(z)\|^{d}}{|Q_{i_{j}}(\tilde{f})(z)|}\cdot\prod^{N}_{j=N-n+\kappa+1}\frac{\|\tilde{f}(z)\|^{d}}{|Q_{i_{j}}(\tilde{f})(z)|}
\\\nonumber&\leq& h_{1}\prod^{\kappa}_{j=1}\frac{\|\tilde{f}(z)\|^{d}}{|P_{Ij}(\tilde{f})(z)|}\cdot\prod^{N-n+\kappa}_{j=\kappa+1}
\frac{\|\tilde{f}(z)\|^{d}}{|Q_{i_{j}}(\tilde{f})(z)|}\cdot\prod^{n}_{j=\kappa+1}\frac{\|\tilde{f}(z)\|^{d}}{|P_{Ij}(\tilde{f})(z)|},
\end{eqnarray}
where $h_{1}=g^{q-N}(z)h^{n-\kappa}(z),$ $I=(i_{1},\ldots,i_{N+1})$ and $\zeta$ is a function in $\mathcal{C}_{f},$ which is chosen common for all $I\in\mathcal{I},$ such that $$|P_{Ij}(\omega)|\leq\zeta(z)\|\omega\|^{d}, \forall\omega=(\omega_0,\ldots,\omega_n).$$
We will consider if $N-n \leq \kappa,$ we have by the inequality above,
\begin{eqnarray}\label{E2.2}
&&\prod^{q}_{i=1}\frac{\|\tilde{f}(z)\|^{d}}{|Q_{i}(\tilde{f})(z)|}\leq h_{1}\prod^{\kappa}_{j=1}\frac{\|\tilde{f}(z)\|^{d}}{|P_{Ij}(\tilde{f})(z)|}\cdot \left(\prod^{N-n}_{j=1}\frac{\|\tilde{f}(z)\|^{d}}{|P_{Ij}(\tilde{f})(z)|}\right)\cdot\prod^{n}_{j=\kappa +1}\frac{\|\tilde{f}(z)\|^{d}}{|P_{Ij}(\tilde{f})(z)|}
\\ \nonumber&\leq& \zeta(z)^{\kappa+n-N}h_{1}\prod^{\kappa}_{j=1}\frac{\|\tilde{f}(z)\|^{d}}{|P_{Ij}(\tilde{f})(z)|}\cdot \left(\prod^{N-n}_{j=1}\frac{\|\tilde{f}(z)\|^{d}}{|P_{Ij}(\tilde{f})(z)|}\right)\left(\prod^{\kappa}_{j=N-n+1}\frac{\|\tilde{f}(z)\|^{d}}{|P_{Ij}(\tilde{f})(z)|}\right)\\ \nonumber&&\cdot\prod^{n}_{j=\kappa +1}\frac{\|\tilde{f}(z)\|^{d}}{|P_{Ij}(\tilde{f})(z)|}
\leq h_{2}\cdot\left(\prod^{\kappa}_{j=1}\frac{\|\tilde{f}(z)\|^{d}}{|P_{Ij}(\tilde{f})(z)|}\right)^{2}\left(\prod^{n}_{j=\kappa +1}\frac{\|\tilde{f}(z)\|^{d}}{|P_{Ij}(\tilde{f})(z)|}\right)^{2}
\\ \nonumber&=& h_{2}\cdot\left(\prod^{n}_{j=1}\frac{\|\tilde{f}(z)\|^{d}}{|P_{Ij}(\tilde{f})(z)|}\right)^{2},
\end{eqnarray}

where $h_{2}=\zeta(z)^{2n-N}h_{1}\in\mathcal{C}_{f},$ however if $N-n\geq \kappa$, we get
\begin{eqnarray}\label{E2.3}
&&\prod^{q}_{i=1}\frac{\|\tilde{f}(z)\|^{d}}{|Q_{i}(\tilde{f})(z)|} \\ \nonumber&\leq&
h_{1}\prod^{\kappa}_{j=1}\frac{\|\tilde{f}(z)\|^{d}}{|P_{Ij}(\tilde{f})(z)|} \left(\prod^{\kappa}_{j=1}\frac{\|\tilde{f}(z)\|^{d}}{|P_{Ij}(\tilde{f})(z)|}\right)^{\frac{N-n}{\kappa}}\prod^{n}_{j=\kappa +1}\frac{\|\tilde{f}(z)\|^{d}}{|P_{Ij}(\tilde{f})(z)|}
\\ \nonumber&=&h_{1}\cdot\left(\prod^{\kappa}_{j=1}\frac{\|\tilde{f}(z)\|^{d}}{|P_{Ij}(\tilde{f})(z)|}\right)^{1+\frac{N-n}{\kappa}}\cdot\prod^{n}_{j=\kappa +1}\frac{\|\tilde{f}(z)\|^{d}}{|P_{Ij}(\tilde{f})(z)|}
\\ \nonumber&\leq& h_{3}\cdot\left(\prod^{\kappa}_{j=1}\frac{\|\tilde{f}(z)\|^{d}}{|P_{Ij}(\tilde{f})(z)|}\right)^{1+\frac{N-n}{\kappa}}\left(\prod^{n}_{j=\kappa +1}\frac{\|\tilde{f}(z)\|^{d}}{|P_{Ij}(\tilde{f})(z)|}\right)^{1+\frac{N-n}{\kappa}}
\\ \nonumber&=& h_{3}\left(\prod^{n}_{j=1}\frac{\|\tilde{f}(z)\|^{d}}{|P_{Ij}(\tilde{f})(z)|}\right)^{1+\frac{N-n}{\kappa}}
\end{eqnarray} where $h_{3}=h_{1}\zeta^{\frac{(n-\kappa)(N-n)}{\kappa}}(z)\in\mathcal{C}_{f}.$

Thus by (\ref{E2.2}) and (\ref{E2.3}),we get
\begin{eqnarray}\label{E2.4}
&&\prod^{q}_{i=1}\frac{\|\tilde{f}(z)\|^{d}}{|Q_{i}(\tilde{f})(z)|}\leq h^{*}\left(\prod^{n}_{j=1}\frac{\|\tilde{f}(z)\|^{d}}{|P_{Ij}(\tilde{f})(z)|}\right)^{1+\frac{N-n}{\max\{1,\min\{N-n,\kappa\}\}}},
\end{eqnarray}where $h^*=\max\{h_2, h_3\}\in\mathcal{C}_{f}.$

Hence, by taking logarithms in the tow sides of (\ref{E2.4}),we can obtain

\begin{eqnarray} \label{E2.5}
&&\log\prod^{q}_{i=1}\frac{\|\tilde{f}(z)\|^{d}}{|Q_{i}(\tilde{f})(z)|}\\ \nonumber&\leq& \log h^{*}+\left(1+\frac{N-n}{\max\{1,\min\{N-n,\kappa\}\}}\right)\log\left(\prod^{n}_{j=1}\frac{\|\tilde{f}(z)\|^{d}}{|P_{Ij}(\tilde{f})(z)|}\right).
\end{eqnarray}\par

Now, for each non-negative integer $L,$ we denote by $V_{L}$ the vector space (over $\mathcal{K}_{\mathcal{Q}}$) consisting of all homogeneous polynomials of degree $L$ in $\mathcal{K}_{\mathcal{Q}}[x_{0},\ldots, x_{n}]$ and the zero polynomial. Denote by $(P_{I1},\ldots, P_{In})$ the ideal in $\mathcal{K}_{\mathcal{Q}}[x_{0},\ldots, x_{n}]$ generated by $P_{I1},\ldots, P_{In}.$

\begin{lemma}\label{L2.15}(See \cite{D}, Proposition 3.3).
Let $\{P_{i}\}^{q}_{i=1}(q\geq n+1)$ be a set of homogeneous polynomials of common degree $d\geq 1$ in $\mathcal{K}_{f}[x_{0},\ldots, x_{n}]$ in weakly general position. Then for any nonnegative integer $L$ and for any $J:=\{j_{1},\ldots,j_{n}\}\subset\{1,\ldots,q\},$ the dimension of the vector space $\frac{V_{L}}{(P_{j_{1}},\ldots,P_{j_{n}})\cap V_{L}}$ is equal to the number of n-tuples $(s_{1},\ldots,s_{n})\in\mathbf{N}^{n}_{0}$ such that $s_{1}+\cdots+s_{n}\leq L$ and $0\leq s_{1},\ldots,s_{n}\leq d-1.$ In particular, for all $L\geq n(d-1),$ we have $$\dim\frac{V_{L}}{(P_{j_{1}},\ldots,P_{j_{n}})\cap V_{L}}=d^{n}.$$
\end{lemma}

Now, for each positive integer $L$ big enough, divided by $d,$ and $\mathbf{i}= (i_{1},\ldots,i_{n})\in\mathbf{N}^{n}_{0}$ with $\sigma(\mathbf{i})=\sum_{j=1}^{n}i_{j}\leq\frac{L}{d},$ we set
$$W_{L,\mathbf{i}}= \sum_{(\mathbf{j})=\left(j_{1}, \ldots, j_{n}\right) \geq(\mathbf{i})} P_{I1}^{j_{1}} \cdots P_{In}^{j_{n}} \cdot V_{L-d\sigma(\mathbf{j})}$$

It is clear that  $W_{L,(0, \ldots, 0)}=V_{L} \text { and } W_{L,\mathbf{i}} \supset W_{L,\mathbf{i}'}$  if $\mathbf{i}<\mathbf{i}'$, so $\{{W_{L,\mathbf{i}}}\}$ is a filtration of $V_{L}.$ For the proof of the above lemma, refer to \cite{Q1}. \par
\begin{lemma}
Let $\mathbf{i}=(i_{1}, \ldots, i_{n}),\mathbf{i}^{'}=(i_{1}^{'}, \ldots, i_{n}^{'}) \in \mathbb{N}_{0}^{n}.$ Suppose that $\mathbf{i}^{'}$ follows $\mathbf{i}$ in the lexicographic ordering and  $d\sigma(\mathbf{i})<L.$ Then
$$\frac{W_{L,\mathbf{i}}}{W_{L,\mathbf{i}'}}\cong\frac{V_{L-d\sigma(\mathbf{i})}}{(P_{j_{1}},\ldots,P_{j_{n}})\cap V_{L-d\sigma(\mathbf{i})}}.$$
\end{lemma}
This lemma yields that \begin{eqnarray}\dim \frac{W_{L,\mathbf{i}}}{W_{L,\mathbf{i}'}}=\dim\frac{V_{L-d\sigma(\mathbf{i})}}{(P_{j_{1}},\ldots,P_{j_{n}})\cap V_{L-d\sigma(\mathbf{i})}}.\end{eqnarray}\par
Fix a number $L$ large enough (chosen later). Set $u=u_{L} :=\dim V_{L}={L+n \choose n}.$ We assume that $$V_{L}=W_{L,\mathbf{i}_{1}} \supset W_{L,\mathbf{i}_{2}} \supset \cdots \supset W_{L,\mathbf{i}_{K}},$$ where $W_{L,\mathbf{i}_{s+1}}$ follows $W_{L,\mathbf{i}_{s}}$ in the ordering and $\mathbf{i}_{K}=\left(\frac{L}{d}, 0, \ldots, 0\right) .$ It is easy to see that $K$ is the number of $n$ -tuples $(i_{1}, \ldots, i_{n})$ with $i_{j} \geq 0$ and $i_{1}+\cdots+i_{n} \leq \frac{L}{d} .$ Then we have $$K={\frac{L}{d}+n \choose n}.$$
For each $k \in\{1, \ldots, K-1\}$ we set $m_{k}^{I}=\dim \frac{W_{L,\mathbf{i}_{k}}}{W_{L,\mathbf{i}_{k+1}}},$ and set $m_{K}^{I}=1.$ Then by Lemma \ref{L2.15}, $m_{k}^{I}$ does not depends on $\{P_{I1}, \ldots, P_{In}\}$ and $k,$ but only on $\sigma(\mathbf{i}_{k}).$ Hence, we set $m_{k} :=m_{k}^{I}.$ We also note that by Lemma \ref{L2.15} $$m_{k}=d^{n}$$ for all $k$ with $L-d\sigma(\mathbf{i}_{k}) \geq n(d-1)$ (it is equivalent to $\sigma(\mathbf{i}_{k}) \leq \frac{L}{d}-n ).$\par

From the above filtration, we may choose a basis $\{\psi_{1}^{I}, \cdots, \psi_{u}^{I}\}$ of
$V_{L}$ such that $$\{\psi_{u-\left(m_{s}+\cdots+m_{K}\right)+1}^{I}, \ldots, \psi_{u}^{I}\}$$ is a basis of $W_{L,\mathbf{i}_{s}}.$ For each $k \in\{1, \ldots, K\}$ and $l \in\left\{u-\left(m_{k}+\cdots+m_{k}\right)+1, \ldots, u-\right.$ $(m_{k+1}+\cdots+m_{K}) \},$ we may write
$$\psi_{l}^{I}=P_{I1}^{i_{1k}} \dots P_{In}^{i_{nk}} h_{l}, \,\,\, \text { where }\left(i_{1 k}, \ldots, i_{n k}\right)=(i)_{k}, h_{ l} \in W_{L-d\sigma(\mathbf{i}_{k})}^{I}.$$ We may choose $h_{l}$ to be a monomial.\par
We have the following estimates: Firstly, we see that $$\sum_{k=1}^{K} m_{k} i_{s k}=\sum_{l=0}^{\frac{L}{d}} \sum_{k|\sigma(\mathbf{i}_{k})=l} m(l) i_{s k}=\sum_{l=0}^{\frac{L}{d}} m(l) \sum_{k|\sigma(\mathbf{i}_{k})=l} i_{s k}.$$
Note that, by the symmetry $\left(i_{1}, \ldots, i_{n}\right) \rightarrow\left(i_{\sigma(1)}, \ldots, i_{\sigma(n)}\right)$ with $\sigma \in S(n), \sum_{k |\sigma(\mathbf{i}_{k})=l}i_{s k}$ does not depend on $s .$ We set $$a :=\sum_{k=1}^{K} m_{k} i_{s k}, \quad \text { which is independent of } s \text { and } I.$$
Then we have
\begin{eqnarray*}
|\psi_{l}^{I}(\tilde{f})(z)|&\leq&|P_{I1}(\tilde{f})(z)|^{i_{1 k}} \cdots|P_{In}(\tilde{f})(z)|^{i_{n k}}|h_{ l}(\tilde{f})(z)|\\
\nonumber&\leq& c_{l}|P_{I1}(\tilde{f})(z)|^{i_{1 k}} \cdots|P_{In}(\tilde{f})(z)|^{i_{n k}}\|\tilde{f}(z)\|^{L-d\sigma(\mathbf{i})_{k}}\\
\nonumber&=&c_{l}\left(\frac{|P_{I1}(\tilde{f})(z)|^{i_{1 k}}}{\|\tilde{f}(z)\|^{d}}\right)^{i_{1 k}} \cdots\left(\frac{|P_{I n}(\tilde{f})(z)|}{\|\tilde{f}(z)\|^{d}}\right)^{i_{n k}}\|\tilde{f}(z)\|^{L}
\end{eqnarray*} where $c_{l} \in \mathcal{C}_{f},$ which does not depend on $f$ and $z.$ Taking the product on both sides of the above inequalities over all $l$ and then taking logarithms, we obtain
\begin{eqnarray}\label{E18}
\nonumber\log \prod_{l=1}^{u}|\psi_{l}^{I}(\tilde{f})(z)| &\leq& \sum_{k=1}^{K} m_{k}\left(i_{1 k} \log \frac{|P_{I 1}(\tilde{f})(z)|}{\|\tilde{f}(z)\|^{d}}+\cdots+ i_{n k} \log \frac{|P_{I n}(\tilde{f})(z)|}{\|\tilde{f}(z)\|^{d}}\right)\\&&+u L \log \|\tilde{f}(z)\|+\log c_{I}.
\end{eqnarray}
where $c_{I}=\prod_{l=1}^{u} c_{ l} \in \mathcal{C}_{f}.$ By (\ref{E18}), it gives
\begin{equation*}
\log \prod_{l=1}^{u}|\psi_{l}^{I}(\tilde{f})(z)| \leq a\left(\log \prod_{i=1}^{n} \frac{|P_{I i}(\tilde{f})(z)|}{\|\tilde{f}(z)\|^{d}}\right)+u L \log \|\tilde{f}(z)\|+\log c_{I},
\end{equation*}
i.e.,
\begin{equation}\label{E19}
a\left(\log \prod_{i=1}^{n} \frac{\|\tilde{f}(z)\|^{d}}{|P_{Ii}(\tilde{f})(z)|}\right) \leq \log \prod_{l=1}^{u} \frac{\|\tilde{f}(z)\|^{L}}{|\psi_{l}^{I}(\tilde{f})(z)|}+\log c_{I}
\end{equation}
set $c_{0}=h^*\prod_{I}(1+c_{I}^{(1+\frac{N-n}{\max\{1,\min\{N-n,\kappa\}\}}) / a}) \in \mathcal{C}_{f}.$\par

Combining \eqref{E19} with \eqref{E2.5}, we obtain that
\begin{equation}\label{E15}
\log \prod_{i=1}^{q} \frac{\|\tilde{f}(z)\|^{d}}{|Q_{i}(\tilde{f})(z)|} \leq \frac{1+\frac{N-n}{\max\{1,\min\{N-n,\kappa\}\}}}{a} \log \prod_{l=1}^{u} \frac{\|\tilde{f}(z)\|^{L}}{|\psi_{l}^{I}(\tilde{f})(z)|}+\log c_{0}.
\end{equation}

We now write $$\psi_{l}^{I}=\sum_{J\in\mathcal{I}_{L}} c_{lJ}^{I} x_{J} \in V_{L}, \quad c_{lJ}^{I} \in \mathcal{K}_{\{Q_{i}\}},$$
where $\mathcal{I}_{L}$ is the set of all $(n+1)$-tuples $J=(i_{0},\ldots, i_{n})$ with $\Sigma^{n}_{s=0}j_{s}=L, x^{J}=x^{j_{0}}_{0}\cdots x^{j_{n}}_{n}$ and $l\in\{1,\ldots, u\}.$ For each $l,$ we fix an index $J^{I}_{l}\in J$ such that $c^{I}_{lJ^{I}_{l}}\not\equiv0.$ Define $$\mu_{lJ}^{I}=\frac{c_{lJ}^{I}}{c_{lJ^{I}_{l}}^{I}}, J\in \mathcal{I}_{L}.$$

Set $\Phi =\{\mu^{I}_{lJ}; I\subset\{1, \ldots, q\}, \sharp I=n, 1 \leq l\leq u, J\in\mathcal{I}_{L}\}.$ Note that $1 \in\Phi.$ Let $B=\sharp\Phi.$ We see that $B\leq u{q \choose n}({L+n \choose n}-1)={q \choose n}({L+n \choose n}-1){L+n \choose n}.$ For each positive integer $l,$ we denote by $\mathcal{L}(\Phi(l))$ the linear span over $\mathbb{C}$ of the set
$$\Phi(l)=\{\gamma_{1}\cdots\gamma_{l};\gamma_{i}\in\Phi\}.$$ It is easy to see that $$\dim\mathcal{L}(\Phi(l))\leq\sharp\Phi(l)\leq{B+l-1 \choose B-1}.$$

We may choose a positive integer $p$ such that
$$p\leq p_{0}:=\left[\frac{B-1}{\log(1+\frac{\epsilon}{3(n+1)(1+\frac{N-n}{\max\{1,\min\{N-n,\kappa\}\}})})}\right]^{2},$$ and  $$\frac{\dim\mathcal{L}(\Phi(p+1))}{\dim\mathcal{L}(\Phi(p))}\leq 1+\frac{\epsilon}{3(n+1)(1+\frac{N-n}{\max\{1,\min\{N-n,\kappa\}\}})}.$$
Indeed, if $\frac{\dim\mathcal{L}(\Phi(p+1))}{\dim\mathcal{L}(\Phi(p))}> 1+\frac{\epsilon}{3(n+1)(1+\frac{N-n}{\max\{1,\min\{N-n,\kappa\}\}})}$ for all $p\leq p_{0},$ we have
$$\dim\mathcal{L}(\Phi(p_{0}+1))\geq (1+\frac{\epsilon}{3(n+1)(1+\frac{N-n}{\max\{1,\min\{N-n,\kappa\}\}})})^{p_{0}}.$$
Therefore, we have
\begin{eqnarray*}
&&\log(1+\frac{\epsilon}{3(n+1)(1+\frac{N-n}{\max\{1,\min\{N-n,\kappa\}\}})})\\&\leq&\frac{\log\dim\mathcal{L}(\Phi(p_{0}+1))}{p_0}\leq\frac{\log{B+p_0 \choose B-1}}{p_0}\\
&=&\frac{1}{p_0}\log\prod^{B-1}_{i=1}\frac{p_0+i+1}{i}<\frac{(B-1)\log(p_0+2)}{p_0}\\
&\leq&\frac{B-1}{\sqrt{p_0}}\leq\frac{(B-1)\log(1+\frac{\epsilon}{3(n+1)(1+\frac{N-n}{\max\{1,\min\{N-n,\kappa\}\}})})}{B-1}\\
&=&\log(1+\frac{\epsilon}{3(n+1)(1+\frac{N-n}{\max\{1,\min\{N-n,\kappa\}\}})}).
\end{eqnarray*}This is a contradiction.\par

We fix a positive integer $p$ satisfying the above condition. Put $s =\dim \mathcal{L}(\Phi(p))$ and $t =\dim \mathcal{L}(\Phi(p +1)).$ Let ${b_{1}, \ldots, b_{t}}$ be an $\mathbb{C}$-basis of $\mathcal{L}(\Phi(p +1))$ such that ${b_{1}, \ldots, b_{s}}$ be a $\mathbb{C}$-basis of $\mathcal{L}(\Phi(p)).$\par

For each $l\in {1, \ldots, u},$ we set
$$\tilde{\psi}_{l}^{I}=\sum_{J\in\mathcal{I}_{L}} \mu_{lJ}^{I} x_{I} .$$
For each $J\in \mathcal{I}_{L},$ we consider homogeneous polynomials $\phi_{J}(x_{0}, \ldots, x_{n}) =x^{J}.$ Let $F$ be a meromorphic mapping of $\mathbb{C}^{m}$ into $\mathbb{P}^{tu-1}(\mathbb{C})$ with a reduced representation $\tilde{F}=(hb_{i}\phi_{J}(\tilde{f}))_{1\leq i\leq t,J\in \mathcal{I}_{L}},$ where $h$ is a nonzero meromorphic function on $\mathbb{C}^{m}.$ We see that $$\parallel N_{h}(r)+N_{1/h}(r)=o(T_{f}(r)).$$
Since $f$ is assumed to be algebraically nondegenerate over $\mathcal{K}_{\mathcal{Q}},$ $F$ is linearly nondegenerate over $\mathbb{C}.$ We see that there exist nonzero functions $c_{1}, c_{2}\in\mathcal{C}_{f}$ such that $$c_{1}|h|\cdot\|\tilde{f}\|^{L}\leq \|\tilde{F}\| \leq c_{2}|h|\cdot\|\tilde{f}\|^{L}.$$
For each $l\in{1, \ldots, u}, 1 \leq i \leq s,$ we consider the linear form $L^{I}_{il}$ in $x^{J}$ such that $$hb_{i}\tilde{\psi}^{I}_{l}(\tilde{f})=L^{I}_{il}(\tilde{F}).$$
Since $f$ is algebraically nondegenerate over $\mathcal{K}_{\mathcal{Q}},$ it is easy to see that $\{b_{i}\tilde{\psi}^{I}_{l}(\tilde{f}); 1 \leq i \leq s, 1 \leq l\leq M\}$ is linearly independent over $\mathbb{C},$ and so is $\{L^{I}_{il}(\tilde{F}); 1 \leq i \leq s, 1 \leq l\leq u\}.$ This yields that $\{L^{I}_{il}; 1 \leq i \leq s, 1 \leq l\leq u\}$ is linearly independent over $\mathbb{C}.$\par

For every point $z$ which is not neither zero nor pole of any $hb_{i}\psi^{I}_{l}(\tilde{f}),$ we also see that
\begin{eqnarray*}
s\log \prod_{i=1}^{u} \frac{\|\tilde{f}(z)\|^{L}}{|\psi^{I}_{l}(\tilde{f})(z)|} &\leq& \log \prod_{
\small\begin{split} 1\leq l\leq u\\ 1\leq i\leq s \small\end{split}} \frac{\|\tilde{F}(z)\|}{|hb_{i}\psi_{l}^{I}(\tilde{f})(z)|}+\log c_{3}\\&=&\log \prod_{\small\begin{split} 1\leq l\leq u\\ 1\leq i\leq s \small\end{split}} \frac{\|\tilde{F}(z)\|\cdot\|L^{I}_{il}\|}{|L^{I}_{il}(\tilde{F})(z)|}+\log c_{4},
\end{eqnarray*}
where $c_{3}, c_{4}$ are nonzero functions in $\mathcal{C}_{f},$ not depend on $f$ and $I$, but on $\{Q_{i}\}^{q}_{i=1}.$ Combining this inequality and (\ref{E15}), we obtain that

\begin{eqnarray}\label{E20}
&&\log \prod_{i=1}^{q} \frac{\|\tilde{f}(z)\|^{d}}{|Q_{i}(\tilde{f})(z)|}\leq\frac{1+\frac{N-n}{\max\{1,\min\{N-n,\kappa\}\}}}{sa}\\\nonumber &&\cdot\left(\max_{I}\log \prod_{\small\begin{split} 1\leq l\leq u\\ 1\leq i\leq s \small\end{split}} \frac{\|\tilde{F}(z)\|\cdot\|L^{I}_{il}\|}{|L^{I}_{il}(\tilde{F})(z)|}+\log c_{4}\right)+\log c_0.
\end{eqnarray}for all $z$ outside an analytic subset of $\mathbb{C}^{m}.$

Since $\tilde{F}$ is linearly nondegenerate over $\mathbb{C},$ according to Proposition \ref{P1.1}, there exists an admissible set $\alpha=(\alpha_{iJ})_{1\leq i\leq t,J\in \mathcal{I}_{L}}$ with $\alpha_{iJ}\in \mathbf{Z}^{m}_{+},\alpha_{iJ} \leq tu -1,$ such that
$$W^{\alpha}(hb_{i} \tilde{\phi}_{J}(\tilde{f}))=\det(\mathcal{D}^{\alpha_{i'j'}}(hb_{i} \tilde{\phi}_{J}(\tilde{f})))\not\equiv0.$$ By Theorem \ref{TB}, we have
\begin{eqnarray}\label{E21}
&&\parallel\int_{S(r)}\max_{I}\left\{\log \prod_{\small\begin{split} 1\leq l\leq u\\ 1\leq i\leq s \small\end{split}} \frac{\|\tilde{F}(z)\|\cdot\|L^{I}_{il}\|}{|L^{I}_{il}(\tilde{F})(z)|}\right\}\\\nonumber &\leq& tuT_{F}(r)-N_{W^{\alpha}(hb_{i} \tilde{\phi}_{J}(\tilde{f}))}(r)+o(T_{F}(r)).
\end{eqnarray}

Integrating both sides of (\ref{E20})and using (\ref{E21}), we obtain that
\begin{eqnarray}\label{E22}
&&qdT_{f}(r)-\sum^{q}_{i=1}N(r,f^{*}Q_i)\leq \frac{tu(1+\frac{N-n}{\max\{1,\min\{N-n,\kappa\}\}})}{sa}T_{F}(r)\\\nonumber &&-\frac{1+\frac{N-n}{\max\{1,\min\{N-n,\kappa\}\}}}{sa}N_{W^{\alpha}(hb_{i} \tilde{\phi}_{J}(\tilde{f}))}(r)+o(T_{F}(r)+T_{f}(r)).
\end{eqnarray}
We can estimate the following quantity where by using the method of S.D.Quang (to see \cite{Q1}), $$\sum^{q}_{i=1}N(r,f^{*}Q_i)-\frac{1+\frac{N-n}{\max\{1,\min\{N-n,\kappa\}\}}}{sa}N_{W^{\alpha}(hb_{i} \tilde{\phi}_{J}(\tilde{f}))}(r),$$ thus we can get
$$\sum^{q}_{i=1}N(r,f^{*}Q_i)-\frac{1+\frac{N-n}{\max\{1,\min\{N-n,\kappa\}\}}}{sa}N_{W^{\alpha}(hb_{i} \tilde{\phi}_{J}(\tilde{f}))}(r)\leq\sum^{q}_{i=1}N^{[tu-1]}(r,f^{*}Q_i).$$

From this inequality and (\ref{E22}) with a note that $T_{F}(r) =LT_{f}(r) +o(T_{f}(r)),$ we have
\begin{eqnarray}\label{E23}
&&(q-\frac{tuL(1+\frac{N-n}{\max\{1,\min\{N-n,\kappa\}\}})}{dsa})T_{f}(r)\\\nonumber&\leq&\sum^{q}_{i=1}\frac{1}{d}N^{[tu-1]}(r,f^{*}Q_i)+o(T_{f}(r)).
\end{eqnarray}

Now we give some estimates for $A, t \text{and} s.$ For each $I_{k}=(i_{1k}, \ldots, i_{nk})$ with $\sigma(\mathbf{i}_{k}) \leq \frac{L}{d}-n,$ we set
$$i_{(n+1)k} = \frac{L}{d}-n -\sum^{n}_{s=1}i_{s}.$$
Since the number of nonnegative integer $p$-tuples with summation $\leq I$is equal to the number of nonnegative integer $(p +1)$-tuples with summation exactly equal to $I\in \mathbf{Z},$ which is $I+n \choose n$, and since the sum below is independent of $s,$ we have
\begin{eqnarray*}
a&=&\sum_{\sigma(\mathbf{i}_{k})\leq\frac{L}{d}}m^{I}_{k}i_{sk}\\
&\leq&\sum_{\sigma(\mathbf{i}_{k})\leq\frac{L}{d}-n}m^{I}_{k}i_{sk}\\
&=&\frac{d^{n}}{n+1}\sum_{\sigma(\mathbf{i}_{k})\leq\frac{L}{d}-n}\sum_{s=1}^{n+1}i_{sk}\\
&=& \frac{d^n}{n+1}{\frac{L}{d} \choose n}(\frac{L}{d}-n)\\
&=&d^{n}{\frac{L}{d} \choose n+1}.
\end{eqnarray*}\par

Now, for every positive number $x \in[0, \frac{1}{(n+1)^{2}}],$ we have
\begin{eqnarray}\label{E24}
(1+x)^{n}&=&1+nx+\sum_{i=2}^{n}{n \choose i}x^{i}\\\nonumber&\leq& 1+nx+\sum_{i=2}^{n}\frac{n^{i}}{i!(n+1)^{2i-2}}x \\\nonumber&\leq& 1+nx+\sum_{i=2}^{n}\frac{1}{i!}x\\\nonumber&\leq& 1+(n+1)x.
\end{eqnarray}
We chose $L =(n +1)d +2(1+\max\{1,\frac{N-n}{\kappa}\})(n +1)^{3}I(\epsilon^{-1})d.$ Then $L$ is divisible by $d$ and we have
\begin{eqnarray}\label{E25}
&&\frac{(n +1)d}{L-(n +1)d}=\frac{(n +1)d}{2(1+\frac{N-n}{\max\{1,\min\{N-n,\kappa\}\}})(n +1)^{3}I(\epsilon^{-1})d}\leq\frac{1}{2(n +1)^{2}}.
\end{eqnarray}

Therefore, using (\ref{E24})and (\ref{E25})we have
\begin{eqnarray*}
&&\frac{uL}{da}\leq\frac{{L+n \choose n}L}{d^{n+1}{\frac{L}{d} \choose n+1}}=\frac{L\cdot(L+1)\cdots(L+n)}{1\cdot2\cdots n}/\frac{(L-nd)\cdot(L-(n-1))\cdots L}{1\cdot2\cdots (n+1)}\\
&&=(n+1)\prod^{n}_{i=1}\frac{L+i}{(L-(n-i+1)d)}<(n+1)(\frac{L}{(L-(n+1)d)})^{n}\\
&&=(n+1)(1+\frac{(n+1)d}{(L-(n+1)d)})^{n} \\
&&<(n+1)(1+\frac{(n+1)^{2}d}{2(1+\frac{N-n}{\max\{1,\min\{N-n,\kappa\}\}})(n +1)^{3}I(\epsilon^{-1})d})\\
&&\leq(n+1)+\frac{(n+1)^{3}d}{2(1+\frac{N-n}{\max\{1,\min\{N-n,\kappa\}\}})(n +1)^{3}\epsilon^{-1}}\\
&&\leq n+1+\frac{\epsilon}{2(1+\max\{1,\frac{N-n}{\kappa}\})}.
\end{eqnarray*}\par

Then we have
\begin{eqnarray}\label{E28}
\frac{tuL}{das}&\leq&(1+\frac{\epsilon}{3(n+1)(1+\frac{N-n}{\max\{1,\min\{N-n,\kappa\}\}})})\\\nonumber
&&\cdot(n+1+\frac{\epsilon}{2(1+\frac{N-n}{\max\{1,\min\{N-n,\kappa\}\}})})\\\nonumber
&\leq& n+1+\frac{\epsilon}{2(1+\frac{N-n}{\max\{1,\min\{N-n,\kappa\}\}})}\\\nonumber
&&+\frac{\epsilon}{3(1+\frac{N-n}{\max\{1,\min\{N-n,\kappa\}\}})}+\frac{\epsilon}{6(1+\frac{N-n}{\max\{1,\min\{N-n,\kappa\}\}})}\\\nonumber
&=&n+1+\frac{\epsilon}{1+\frac{N-n}{\max\{1,\min\{N-n,\kappa\}\}}}.
\end{eqnarray}\par

Combining (\ref{E23})and (\ref{E28}), we get
\begin{eqnarray}\label{E26}
&&(q-(1+\frac{N-n}{\max\{1,\min\{N-n,\kappa\}\}})(n+1)-\epsilon)T_{f}(r)\\\nonumber&&\leq\sum^{q}_{i=1}\frac{1}{d}N^{[tu-1]}(r,f^{*}Q_i)+o(T_{f}(r)).
\end{eqnarray}
Here we note that:\par
$$L:=(n +1)d +2(1+\frac{N-n}{\max\{1,\min\{N-n,\kappa\}\}})(n +1)^{3}I(\epsilon^{-1})d,$$\par
\begin{eqnarray*}
p_{0}&:=&\left[\frac{B-1}{\log(1+\frac{\epsilon}{3(n+1)(1+\frac{N-n}{\max\{1,\min\{N-n,\kappa\}\}})})}\right]^{2}\\&\leq&\left[\frac{{L+n \choose n}({L+n \choose n}-1){q \choose n}-1}{\log(1+\frac{\epsilon}{3(n+1)(1+\frac{N-n}{\max\{1,\min\{N-n,\kappa\}\}})})}\right]^{2},
\end{eqnarray*}
\begin{eqnarray*}
tu-1&\leq&{L+n \choose n}{B+p \choose B-1}-1\leq{L+n \choose n}p^{B-1}-1\\ &\leq&{L+n \choose n}p_{0}^{{L+n \choose n}{L+n \choose n}-1){q \choose n}-2}-1=L_{0}.
\end{eqnarray*}

By these estimates and by (\ref{E26}), we obtain
\begin{eqnarray}\label{E27}
&&(q-(1+\frac{N-n}{\max\{1,\min\{N-n,\kappa\}\}})(n+1)-\epsilon)T_{f}(r)\\\nonumber&&\leq\sum^{q}_{i=1}\frac{1}{d}N^{[L_0]}(r,f^{*}Q_i)+o(T_{f}(r)).
\end{eqnarray}
The theorem is proved.\vskip 6pt

\noindent{\bf Acknowledgement.} The authors are very grateful to Prof. Qiming Yan for pointing out a hard to find but serious gap in the original version (arXiv:1908.05844v1) and giving many valuable suggestions, and would also thank to Dr. Nguyen Van Thin for giving some comments to the first version (arXiv:1908.05844v1).

\end{document}